\documentclass[11pt]{amsart}

\newcommand\embed{\hookrightarrow}

\usepackage{amsmath}
\usepackage{amsfonts}
\usepackage{mathtools}
\usepackage{amssymb}
\usepackage{amsthm}
\usepackage{mathtools}
\usepackage{latexsym}
\usepackage{fouridx}

\usepackage{enumerate}

\usepackage{wrapfig}
\usepackage{caption}

\usepackage{mathrsfs}
\usepackage[T1,T2A]{fontenc}
\usepackage[utf8]{inputenc}
\usepackage[russian,english]{babel}

\usepackage{verbatim}

\usepackage{stmaryrd}
\SetSymbolFont{stmry}{bold}{U}{stmry}{m}{n}

\usepackage{pst-node}
\usepackage{tikz-cd} 
\usetikzlibrary{quotes,babel}

\usepackage{color}

\newcommand{\duality}[2]{\fourIdx{}{(V^s)^\prime}{}{V^s}{\langle #1, #2 \rangle}}

\newcommand\dela[1]{}

\usepackage{float}

\usepackage{color,stmaryrd,bbm}
\usepackage{fullpage}

\theoremstyle{plain}
\newtheorem{theorem}{Theorem}[section]
\newtheorem{proposition}[theorem]{Proposition}
\newtheorem{lemma}[theorem]{Lemma}

\theoremstyle{definition}
\newtheorem{definition}[theorem]{Definition}

\theoremstyle{remark}
\newtheorem{remark}[theorem]{Remark}

\numberwithin{equation}{section} 
\numberwithin{figure}{section}   

\usepackage[square,comma,numbers,sort]{natbib} 

\newcommand{\vect}[1]{\mathbf{#1}}

\newcommand{\bu}{\vect{u}}
\newcommand{\bv}{\vect{v}}
\newcommand{\bw}{\vect{w}}

\newcommand{\bx}{\vect{x}}

\newcommand{\bbf}{\vect{f}}

\newcommand{\bomega}{\boldsymbol{\omega}}
\newcommand{\field}[1]{\mathbb{#1}}
\newcommand{\nN}{\field{N}}
\newcommand{\nZ}{\field{Z}}

\newcommand{\nR}{\field{R}}

\newcommand{\cD}{\mathcal D}

\newcommand{\cV}{\mathcal V}

\newcommand{\nT}{\mathbb T}
\newcommand{\vphi}{\boldsymbol{\varphi}}
\newcommand{\maps}{\rightarrow}

\newcommand{\lp}{\left(}
\newcommand{\rp}{\right)}
\newcommand{\inv}{^{-1}}

\newcommand{\abs}[1]{\left\lvert#1\right\rvert}
\newcommand{\set}[1]{\left\{#1\right\}}
\newcommand{\ip}[2]{\left<#1,#2\right>}

\newcommand{\pnt}[1]{\left(#1\right)}

\newcommand{\norm}[1]{\left\|#1\right\|}

\title[Fractional Voigt]{Fractional Voigt-regularization of the 3D Navier--Stokes and Euler equations: Global well-posedness and limiting behavior}

\date{\today}

\author{Zdzis{\l}aw Brze{\'z}niak}
\address[Zdzis{\l}aw Brze{\'z}niak]{Department of Mathematics, 
University of York,
York,
YO10 5DD, UK}
\email[Zdzislaw Brzezniak]{zdzislaw.brzezniak@york.ac.uk}
\author{Adam Larios}
\address[Adam Larios]{Department of Mathematics, 
                University of Nebraska--Lincoln,
        Lincoln, NE 68588-0130, USA}
\email[Adam Larios]{alarios@unl.edu}
\author{Isabel Safarik}
\address[Isabel Safarik]{Cooperative Institute for Research in the Atmosphere (CIRA), Colorado State University, NOAA/OAR Global Systems Laboratory, Boulder, CO 80305, USA}
\email[Isabel Safarik]{isabel.safarik@colostate.edu}



\begin{document}
\begin{abstract}
The Voigt regularization is a technique used to model turbulent flows, offering advantages such as sharing steady states with the Navier-Stokes equations and requiring no modification of boundary conditions; however, the parabolic dissipative character of the equation is lost. In this work we propose and study a generalization of the Voigt regularization technique by introducing a fractional power $r$ in the Helmholtz operator, which allows for dissipation in the system, at least in the viscous case.  We examine the resulting fractional Navier-Stokes-Voigt (fNSV) and fractional Euler-Voigt (fEV) and show that global well-posedness holds in the 3D periodic case for fNSV when the fractional power $r \geq \frac{1}{2}$ and for fEV when $r>\frac{5}{6}$. Moreover, we show that the solutions of these fractional Voigt-regularized systems converge to solutions of the original equations, on the corresponding time interval of existence and uniqueness of the latter, as the regularization parameter $\alpha \to 0$. Additionally, we prove convergence of solutions of fNSV to solutions of fEV as the viscosity $\nu \to 0$ as well as the convergence of solutions of fNSV to solutions of the 3D Euler equations as both $\alpha, \nu \to 0$. Furthermore, we derive a criterion for finite-time blow-up for each system based on this regularization.  These results may be of use to researchers in both pure and applied fluid dynamics, particularly in terms of approximate models for turbulence and as tools to investigate potential blow-up of solutions.
\end{abstract}

\maketitle
\thispagestyle{empty}

\section{Introduction}\label{secInt}

\noindent
One of the defining characteristics of turbulent fluid flow is the large range of spatial and temporal scales which can be seen physically as cascading eddies. This characteristic is a source of difficulty in both analysis and computation. Moreover, in many practical applications, the physically relevant aspects of the flow often involve the large scales, as seen, for instance, in weather prediction. Because of this, researchers have put effort into modeling the large-scale motion of the dynamics for turbulent flows by filtering out the smaller scales. One candidate for a subgrid-scale model, from \cite{Cao_Lunasin_Titi_2006}, is posed by applying the Helmholtz operator $(I - \alpha^2\triangle)$ to the time derivative of the Navier-Stokes equations:
\begin{subequations}\label{NSV}
\begin{alignat}{2}
\label{NSV_mo}
 (I - \alpha^2\triangle)\partial_t\bv + (\bv \cdot \nabla) \bv  + \nabla p &= \nu \triangle \bv + \bbf, \ \ \ & \text{ in } \Omega \times [0,T),\\
\label{NSV_div_free}
 \nabla \cdot \bv &= 0, & \text{ in } \Omega \times [0,T), \\
\label{NSV_IC}
 \bv(0) &= \bv^0, & \text{ in } \Omega,
\end{alignat}
\end{subequations}
where the boundary conditions are taken to be periodic on the torus $\Omega := \nT^3 = \nR^3 / \nZ^3 = [0,1]^3$ and $\bv$ is mean-free.
Note that the Navier-Stokes equations are formally recovered by setting $\alpha = 0$. When $\alpha, \nu >0$, we obtain the Navier-Stokes-Voigt\footnote{Sometimes incorrectly spelled ``Voight,'' likely due to transliteration inconsistencies when translating the German name ``Woldemar Voigt'' into Russian as  
``\foreignlanguage{russian}{Вольдемар Фогт}'' (Voldemar Fogt), as in, e.g., the Russian bibliography of W. Voigt in \cite{Khramov_1983}, and then re-translated back into English with the Russian ``\foreignlanguage{russian}{г}'' transliterated as the English ``gh.''
}(NSV) equations. 

If $\alpha >0$ and we set $\nu = 0$, we formally obtain the Euler-Voigt equations under periodic boundary conditions.
A noteworthy observation made in \cite{Cao_Lunasin_Titi_2006} is that the NSV equations behave like a damped hyperbolic system, see \cite{Kalantarov_Titi_2009}. Moreover, the instantaneous smoothing of initial data present in the Navier-Stokes equations is no longer present in NSV. This fact could imply that the NSV model is not a reasonable modification of the Navier-Stokes equations. The focus of this paper lies in examining a modification to the Voigt-regularization technique that can introduce fractional dissipation into the system. In particular, we explore the incorporation of a (fractional) power in the Helmholtz operator. We propose the following modification of the 3D Navier-Stokes equations, which we will call the fractional Navier-Stokes-Voigt (fNSV) equations: 
\begin{subequations}\label{fNSV}
\begin{alignat}{2}
\label{fNSV_mo}
(I+ (-\alpha^2\triangle)^r)\partial_t\bv + (\bv \cdot \nabla) \bv  + \nabla p &= \nu \triangle \bv + \bbf, \ \ \ & \text{ in } \Omega \times [0,T),\\
\label{fNSV_div_free}
\nabla \cdot \bv &= 0, & \text{ in } \Omega \times [0,T), \\
\label{fNSV_IC}
\bv(0) &= \bv^0, & \text{ in } \Omega,
\end{alignat}
\end{subequations}
where the parameter $r>0$ is the fractional exponent, boundary conditions are again periodic, and $\bv^0$ and $\bbf$ are mean-free (note that this implies that $\bv$ is mean-free). Global well-posedness was proved in the case $r = 1$ in \cite{Cao_Lunasin_Titi_2006, Oskolkov_1973, Oskolkov_1976, Oskolkov_1982}. In the present work, we show that global well-posedness also holds in the 3D case when $r>\frac{5}{6}$ if $\nu = 0$, and $r \geq \frac{1}{2}$ if $\nu >0$. 

A comparable technique has been explored in the context of the 2D Boussinesq equations in \cite{Ignatova_2023_BoussVoigt}. However, the application of power differs, in particular, the fractional Helmholtz operator is posed as $(I + \epsilon (- \triangle)^{1/2})^r$. Additionally, the power is maintained within more stringent constraints $r >1$. In terms of three dimensions, a similar concept was examined in \cite{Hani_2013}, but the model presented in that work is formulated in a deconvolutional manner. Furthermore, it is crucial to highlight that although a blow-up criterion was postulated in \cite{Hani_2013}, no convergence was demonstrated, a crucial aspect (to the best of our knowledge) for substantiating the proof of the blow-up criterion. In this paper, we address this gap by incorporating the necessary convergence analysis to support the asserted blow-up criterion.

The system \eqref{NSV} was first posed in 1973 by A.P. Oskolkov \cite{Oskolkov_1973, Oskolkov_1976, Oskolkov_1982} as a model for polymeric fluids (Kelvin-Voigt fluids). The parameter $\alpha$ introduces another length scale which is physically meaningful in this setting. In particular, $\alpha \geq 0$ is a given length scale parameter such that $\alpha^2 / \nu$ is the relaxation time of a viscoelastic fluid. However, system \eqref{NSV} is more commonly known as a regularization of the Navier-Stokes equations, as reintroduced in 2006 by Y. Cao, E. Lunasin, and E.S. Titi \cite{Cao_Lunasin_Titi_2006}, in which $\alpha$ is thought of as a regularization parameter. 

The Helmholtz operator $(I - \alpha^2 \triangle)$ was chosen for its use in the $\alpha$ models of turbulence (see e.g. \cite{Chen_Foias_Holm_Olson_Titi_Wynne_1998_PRL,   Cheskidov_Holm_Olson_Titi_2005, Foias_Holm_Titi_2002, Foias_Manley_Rosa_Temam_2001}), namely the NS-$\alpha$ model, as it demonstrated successful comparisons with experimental data \cite{Chen_Foias_Holm_Olson_Titi_Wynne_1998_PF, Chen_Foias_Holm_Olson_Titi_Wynne_1999}. Notably, unlike other $\alpha$-models, when the viscosity parameter $\nu$ is set to zero, the Euler-Voigt equations---also referred to as the inviscid simplified Bardina model due to contributions in \cite{Layton_Lewandowski_2006}---serve as a regularization for the Euler equations. Another advantage of the NSV model is that, in a bounded domain, the NSV equations do not require any additional (i.e. artificial) boundary conditions. This is particularly notable in the case of fNSV when $r > 1$.

Key areas of research and advancements for Voigt regularization include global well-posedness, higher order regularity, the existence of global and pullback attractors, stability, numerical methods, order reduction, and applications. These things have been explored in the context of a wide variety of nonlinear equations including 
the Navier-Stokes equations \cite{Larios_Titi_2009,
Larios_Pei_Rebholz_2018, Kalantarov_Levant_Titi_2009, Kalantarov_Titi_2009, Garcia_Luengo_Julia_Read_2012, Li_Qin_2013, Berselli_Bisconti_2012, Levant_Ramos_Titi_2009, Berselli_Spirito_2017, Borges_Ramos_2013, Ebrahimi_Holst_Lunasin_2012, Gal_Medjo_2013_MMAS, Yuming_Qin_2025_DCDS},
the magnetohydrodynamic (MHD) equations \cite{Catania_2009, Catania_Secchi_2009, Larios_Titi_2010_MHD, Pei_2021}, 
the surface quasi-geostrophic (SQG) equations \cite{Khouider_Titi_2008}, and the
Boussinesq equations \cite{ Larios_Lunasin_Titi_2013, Bisconti_Catania_2021}, but by no means is this a comprehensive list. 
Numerical methods for a Voigt regularization have been explored using finite element methods for the Navier-Stokes and MHD equations \cite{Kuberry_Larios_Rebholz_Wilson_2012, Rong_Fiodilino_Shi_Cao_2022, Layton_Rebholz_2013_Voigt, Cuff_Dunca_Manica_Rebholz_2015}, 
as well as with pseudo-spectral methods for the Euler equations \cite{DiMolfetta_Krstlulovic_Brachet_2015, Larios_Petersen_Titi_Wingate_2015}. 
See also \cite{Berselli_Kim_Rebholz_2016, Gao_Sun_2012, Niche_2016_JDE, Ramos_Titi_2010, Tang_2014, Zhao_Zhu_2015, Zang_2022, Bohm_1992} and the references therein. 

We also mention several related works. Namely, \cite{Berselli_Kim_Rebholz_2016} compares a reduced order approximate deconvolutional model with Voigt model;
\cite{Gao_Sun_2012, Tang_2014} discuss Stochastic NSV;
\cite{Niche_2016_JDE, Zhao_Zhu_2015} consider long term behavior of NSV solutions;
\cite{Ramos_Titi_2010} shows statistical properties using the notation of invariant measures;
\cite{Zang_2022} analyzes the boundary layer for Euler-Voigt; and
\cite{Bohm_1992} shows existence of weak solutions of the NSV equations in intermediate spaces between some of the usual spaces for weak solutions of the Navier-Stokes equations.

In the limiting case as $\alpha \to 0$, solutions of the NSV and Euler-Voigt equations converge to solutions of the Navier-Stokes and Euler equations, respectively. Because of this convergence result, the Voigt regularization has been used in the context of shock-capturing specifically through the use of a blow-up criterion 
\cite{Gibbon_Titi_2013_Blowup,
 Larios_Pei_Rebholz_2018, Larios_Titi_2015_BC_Blowup, Larios_Petersen_Titi_Wingate_2015, Larios_Titi_2009, Khouider_Titi_2008}.
In this paper, we prove similar results for fractional cases.

The study of blow-up criteria for the Euler and Navier-Stokes equations has evolved significantly over the years. The seminal work \cite{Beale_Kato_Majda_1984} established a key blow-up criterion for the 3D Euler equations by proving that singularity formation is contingent upon the time-integrated vorticity becoming unbounded. Specifically, a solution to the 3D Euler equations can develop a singularity in finite time if and only if the integral of the maximum vorticity over time diverges. This result links the blow-up to the behavior of the vorticity field, providing a precise condition under which solutions can break down. The result was extended to the deformation tensor in the max norm \cite{Ponce_1985} and later to bounded domains \cite{Ferrari_1993}. The direction of vorticity in blow-up scenarios was explored in \cite{Constantin_Fefferman_1993}. Another blow-up criterion for the Euler equations was provided in \cite{Kozono_Taniuchi_2000}. A unified blow-up framework \cite{Cheskidov_Shvydkoy_2014_unified_blow_up} helped consolidate these criteria, with further improvements made in subsequent works. On the computational side, early studies suggested that vortex stretching could lead to finite-time blow-up \cite{Morf_Orszag_Frisch_1980_computational_blow_up}, supported by later works \cite{Brachet_Meiron_Orszag_Nickel_Morf_Frisch_1983_JFM,Kerr_1993}. Other simulations examined vorticity moments and regularity in Navier-Stokes flows \cite{Donzis_Gibbon_Gupta_Kerr_Pandit_Vincenzi_2013,Gibbon_Bustamante_Kerr_2008,Bustamante_Brachet_2012}. Sharp analyses of potential blow-up times in MHD flows were explored in \cite{Brachet_Bustamante_Krstulovic_Mininni_Pouquet_Rosenberg_2013}, while simulations of potential Euler singularities were presented in \cite{Luo_Hou_2014_Euler_BlowUp,Luo_Hou_2013_Potentially_Singular}. Sharper estimates on supremum norms and singularity times were achieved through simulations of a one-parameter family of models \cite{Mulungye_Lucas_Bustamante_2015}.

We note that in the recent arXiv preprint \cite{Chen_Hou_2022_Euler_Blow_up} has claimed a proof of blow-up for the 3D Euler equations in a case involving a bounded domain, but we do not comment further on this work as it has not been published at the time of this writing.  However, even if such a result were true, blow-up in the case of periodic boundary conditions does not automatically follow: see, e.g., \cite{Larios_Titi_2015_BC_Blowup} which pointed out that the PDE $u_t = \triangle u + |\nabla u|^4$ blows up in finite time under Dirichlet boundary conditions, but is globally well-posed in the case of periodic boundary conditions.  Hence, showing blow-up of the 3D incompressible Euler equations still remains an extremely challenging open problem, at least in the periodic case.  

This paper is organized as follows: in Section \ref{sec:Preliminaries} we lay out notation and state definitions and preliminary theorems and results for reference; in Section \ref{sec:EV_GWP} we prove global well-posedness for the fractional Euler-Voigt problem; in Section \ref{sec:NSV_GWP}, we consider the viscous case and prove global well-posedness for the fractional NSV problem; in Section \ref{sec:Convergence} we prove convergence as $\alpha \to 0$ for both fixed $\nu > 0$ and $\nu = 0$ as well as show convergence as $\alpha, \nu \to 0$ at the same time; in Section \ref{sec:Blow_up}, we pose a blow-up criterion for the Navier-Stokes and Euler equations, and concluding remarks are in the last section.

\section{Preliminaries}\label{sec:Preliminaries}
\noindent
In this section, we lay out our notation and state some basic results. Let $\Omega$ be the unit torus, i.e., $\nT^3 = \nR^3 \setminus \nZ^3 = [0,1]^3$. The vector $\bu = (u_1(\bx,t), u_2(\bx,t), u_3(\bx,t))$ denotes the velocity field of a fluid at position $\bx = (x_1, x_2, x_3) \in \Omega$ and time $t \in [0, T)$ where $T>0$. Pressure is denoted by the scalar quantity $p  = p(\bx,t)$, the body force on the fluid is given by $\bbf = \bbf(\bx,t) = (f_1(\bx,t), f_2(\bx,t), f_3(\bx,t))$, and $\nu >0$ is the kinematic viscosity. The physical no-slip case is represented by the Dirichlet boundary conditions $\bv \big|_{\partial \Omega} = \boldsymbol{0}$. Nevertheless, the equations, when subjected to periodic boundary conditions, retain numerous---though not all---structures found in the Dirichlet scenario. 

We define the relevant spaces with respect to the periodic boundary conditions on the unit torus $\nT^3$. With this in mind, we denote the following divergence-free space of smooth functions
\begin{subequations}
\begin{align}
\mathcal{V} &:=\set{\bu\in C^\infty(\nT^3;\nR^3):\nabla\cdot\bu= 0 \text{ and }\int_\Omega\bu\,d\bx= \boldsymbol{0}}.
\end{align}
\end{subequations}
Let $L^p$ and $H^s$ denote the usual Lebesgue and Sobolev spaces on either $\Omega$ or $\nT^3$. 
We denote the (real) $L^2$ inner-product and $H^s$ Sobolev norm by 
\begin{align}\label{Hs_norm}
    (\bu , \bv) : = \sum\limits_{j = 1}^3 \int_\Omega u_j (\bx) v_j(\bx) \, d\bx, \ \ \ \ \ \norm{\bu}_{H^s} := \lp \sum\limits_{|\alpha|\leq m} \norm{D^\alpha \bu}_{L^2}^2 \rp ^{1/2},
\end{align}
here $\alpha = (\alpha_1, \alpha_2, \alpha_3)$ and $D^\alpha \bu = \partial_1^{\alpha_1}\partial_2^{\alpha_2}\partial_3^{\alpha_3}$. For brevity, we will use the notation $L^2(\Omega) \equiv L^2$ and $H^s(\Omega)\equiv H^s$ throughout the paper. 
Using the standard Lions notation, we denote $H$ and $V$ to be the closures of $\mathcal{V}$ in the $L^2$ and $H^1$ norms respectively. Furthermore, we denote $V^s$ to be the closure of $\mathcal{V}$ in the $H^s$ norm for $s > 0$. Moreover, if we put $s = 1$, then $V^1 =: V$. 

Below, we state several standard results about $A$, all of which can be found in, e.g., \cite{Constantin_Foias_1988,Doering_Gibbon_1995_book,Temam_2001_Th_Num}.
We denote by $P_\sigma:L^2\maps H$ the Leray-Helmholtz projection operator onto divergence-free space and denote the Stokes operator 
\begin{align}
 A:=-P_\sigma\triangle,
\end{align}    
\begin{align}
\mathcal{D}(A):=H^2(\nT^3)\cap H,
\end{align}
where $\triangle$ is the Laplace operator in the Hilbert space $L^2(\mathbb{T}^3,\mathbb{R}^3)$.
It is known that $A$ is  strictly positive self-adjoint and invertible operator in $H$. In particular, 
for $\vphi\in \mathcal{D}(A)$, we have the norm equivalence $\|A\vphi\|_{L^2}\cong\|\vphi\|_{H^2}$. 
The inverse $A\inv$ is a compact operator in $H$.  Therefore, there exists an orthonormal basis of eigenfunctions of $A$, say $\set{\vphi_k}_{k\in\nN}$, with corresponding positive eigenvalues $\set{\mu_k}_{k\in\nN}$ ordered to be decreasing.  Setting $\lambda_k:=1/\mu_k$, we have $A\vphi_k=\lambda_k\vphi_k$ for each $k\in\nN$, and $0<\lambda_1\leq \lambda_2\cdots$. Moreover, $\lambda_k\rightarrow\infty$ as $k\rightarrow\infty$.  The fractional Stokes operator, i.e. the fractional power of the Stokes operator $A$  is denoted by $A^s $. It satisfies 
\begin{align}
A^s \bu &= \sum\limits_{j=1}^\infty \lambda_j^s u_j \vphi_j \text{ for } \bu = \sum\limits_{j=1}^\infty u_j \vphi_j, \ \bu \in \cD(A^s),
\end{align}
where $s>0$, and the domain of $A^s$ is given by 
\begin{align}\notag
    \cD(A^s) := H^{2s}(\nT^3)\cap H = \set{\bu \in H : \bu = \sum\limits_{j=1}^\infty u_j \vphi_j \text{ and } \sum\limits_{j=1}^\infty \lambda_j^{s} \abs{u_j}^2 < \infty \text{ for all } u_j \in \nR}.
\end{align}
By Parseval's Identity, the $H^s$ norm \eqref{Hs_norm} is equivalently 
\begin{align}\label{Hs_norm_parseval}
    \norm{\bu}_{H^s}^2 \equiv \sum\limits_{j=1}^\infty \lambda_j^s \abs{u_j}^2.
\end{align}
It is well known that the operator $-\triangle$ has similar properties as the Stokes operator $A$. In particular it is non-negative and self-adjoint so that its fractional powers are well defined. 

Applying $P_\sigma$ to \eqref{fNSV}, it can be shown that the functional form of the fractional Navier-Stokes-Voigt equations are given by 
\begin{align}\notag 
    (I + P_\sigma(-\alpha^2\triangle)^r)\frac{d}{d t} \bu + B(\bu,\bu) + \nu A\bu  = \vect{0},
\end{align}
where we use the conventional notation 
\begin{align}\label{B_operator}
    B(\bu, \bv) : = P_\sigma((\bu \cdot \nabla) \bv),
\end{align}
for $\bu, \bv \in V$.
However, on a periodic domain, it is equivalent (and notionally simpler) to instead write 
\begin{align}\label{fNSV_functional}
    (I + \alpha^{2r} A^r)\frac{d}{d t} \bu + B(\bu,\bu) + \nu A\bu = \vect{0}.
\end{align}
Adding a sufficiently smooth, mean-free body force $\bbf$ causes no additional difficulty in the proof, however for the sake of brevity we do not include it in the following work.
\begin{lemma}
On a periodic domain, we have 
\begin{equation}\label{def_fractional_Stokes}\notag 
  P_\sigma(-\triangle)^r=  A^r.  
\end{equation}
\end{lemma}
\begin{proof}
See, e.g., \cite{Constantin_Foias_1988}. 
\end{proof}

We list several important properties of $B$ (defined by \eqref{B_operator}). Similar results can be found for example in \cite{Constantin_Foias_1988, Foias_Manley_Rosa_Temam_2001, Temam_1995_Fun_Anal, Temam_2001_Th_Num}.

Moreover, by identifying the dual space $H^\prime$  with the space  $H$,   the dual space $H^\prime$ can be identified with a subspace of ${V}^\prime$ and hence we have the following  Gelfand triple
\begin{equation} \label{eqn-Gelfand triple}
 V \embed  H \cong H^\prime   \embed V^{\prime } .
\end{equation}
A similar property hold for $V^s$.  If $s>1$, then we have 
\begin{equation} \label{eqn-Gelfand triple-s1}
 V^s \embed  V \embed  H \cong H^\prime   \embed V^{\prime } \embed (V^s)^{\prime },
\end{equation}
while, when $s\in (0,1)$ we have 
\begin{equation} \label{eqn-Gelfand triple-s2}
 V \embed   V^s \embed  H \cong H^\prime   \embed (V^s)^{\prime } \embed V^{\prime }.
\end{equation}

\begin{lemma}\label{lemma:bilinear_estimates}
The operator $B$ defined in \eqref{B_operator} is a bilinear form which can be extended as a continuous map $B : V^s \times V^s \to (V^s)'$ for $s \geq \frac12$.
\begin{enumerate}[(i)]
    \item For $\bu, \bv \in V^s$

    \begin{align} \label{switch}
        \duality{B(\bu, \bv)}{\bw} = - \duality{B(\bu, \bw)}{\bv},
    \end{align}
    and therefore 
    \begin{align}
  \duality{B(\bu, \bv)}{\bv}     
  = 0.
    \end{align}

    \item Furthermore, we have the following estimates for all $\bu, \bv, \bw \in V^s$ and $0 < \epsilon \leq \frac{6s - 5}{2}$: 
    \begin{align}
        \label{r_52_2r_r}
        \abs{\duality{B(\bu, \bv)}{\bw}} 
        & \leq 
        C \norm{\bu}_{H^s}\norm{\bv}_{H^{\frac{5}{2} - 2s}} \norm{\bw}_{H^s}
        \\ & \leq \notag
        C \norm{\bu}_{H^s}\norm{\bv}_{H^{s}} \norm{\bw}_{H^s},
        & \text{ for } \ \frac{5}{6} \leq s \leq \frac{3}{2},
        \\ \label{r_r_r_epsilon} 
        \abs{\duality{B(\bu, \bv)}{\bw}} 
        & \leq 
        C \norm{\bu}_{H^s}\norm{\bv}_{H^s} \norm{\bw}_{H^{s-\epsilon}}
        \\ & \leq \notag
        C \norm{\bu}_{H^s}\norm{\bv}_{H^s} \norm{\bw}_{H^{s}},
        & \text{ for } \  \frac{5}{6} < s \leq \frac{5}{4},
        \\  \label{r_2_32_r}
        \abs{\duality{B(\bu, \bv)}{\bw}} 
        & \leq 
        C \norm{\bu}_{H^s}\norm{\nabla \bv}_{L^2} \norm{\bw}_{H^{\frac{3}{2} - s}} 
        \\ & \leq \notag
        C \norm{\bu}_{H^s}\norm{\nabla \bv}_{L^2} \norm{\bw}_{H^{1}}, 
        & \text{ for } \ \frac{1}{2} \leq s \leq \frac{3}{2}.
    \end{align}
\end{enumerate}
\end{lemma}

Although Lemma \ref{lemma:bilinear_estimates} is a special case of Proposition 6.1 in \cite{Constantin_Foias_1988}, we provide proof here for the convenience of the reader. 

\begin{proof}[Proof of (ii)]
For \eqref{r_52_2r_r}, set $p = \frac{6}{3-2s}$ and $q = \frac{3}{2s}$. Then using the Gagliardo-Nirenberg Theorem (see, e.g., \cite{Brezis_2010_functional}), we obtain
\begin{align}
    \abs{\duality{B(\bu, \bv)}{\bw}} 
    & \leq 
    \norm{\bu}_{L^p}\norm{\nabla \bv}_{L^q} \norm{\bw}_{L^p}
    \\ & \leq \notag
    C \norm{\bu}_{H^s}\norm{\nabla \bv}_{H^{\frac{3}{2} - 2r}} \norm{\bw}_{H^s}
    \\ & \leq \notag
    C \norm{\bu}_{H^s}\norm{\bv}_{H^{\frac{5}{2} - 2s}} \norm{\bw}_{H^s}.
\end{align}
Moreover, for $s \geq \frac{5}{6}$, $\norm{\bv}_{H^{\frac{5}{2}-2s}} \leq \norm{\bv}_{H^s}$. 

For \eqref{r_r_r_epsilon}, fix $0 < \epsilon \leq \frac{6s - 5}{2}$ and set $p = \frac{6}{-2+4s-2\epsilon}$, $q = \frac{6}{3 - 2(s-1)}$, and $k = \frac{6}{3 - 2(s- \epsilon)}$. Then 
\begin{align*}
    \abs{\duality{B(\bu, \bv)}{\bw}} 
    & \leq 
    \norm{\bu}_{L^p}\norm{\nabla \bv}_{L^q} \norm{\bw}_{L^k}.
\end{align*}
We see that $\norm{\bu}_{L^p} \leq \norm{\bu}_{H^r}$, $\norm{\nabla \bv}_{L^q} \leq C\norm{\nabla \bv}_{H^{s-1}}$, and $\norm{\bw}_{L^k}\leq C\norm{\bw}_{H^{s-\epsilon}}$ for $r = 3 \lp \frac{1}{2} - \frac{1}{p}\rp = \frac{5 - 4s + 2\epsilon}{2}$ which is no more than $s$ for $0 < \epsilon \leq \frac{6s -5}{2}$. Hence, we obtain
\begin{align}
    \abs{\duality{B(\bu, \bv)}{\bw}} 
    & \leq 
    \norm{\bu}_{L^p}\norm{\nabla \bv}_{L^q} \norm{\bw}_{L^k}
    \\ & \leq \notag 
    C \norm{\bu}_{H^s}\norm{\bv}_{H^s} \norm{\bw}_{H^{s-\epsilon}}.
\end{align}

Finally, \eqref{r_2_32_r} is justified by taking $p = \frac{6}{3-2s}$ and $q = \frac{3}{s}$, to obtain 
\begin{align}
    \abs{\duality{B(\bu, \bv)}{\bw}} 
    & \leq 
    \norm{\bu}_{L^p}\norm{\nabla \bv}_{L^2} \norm{\bw}_{L^q} 
    \\ & \leq \notag
    C \norm{\bu}_{H^s}\norm{\nabla \bv}_{L^2} \norm{\bw}_{H^{\frac{3}{2} - s}}.
\end{align}
Moreover, for $s \geq \frac{1}{2}$, it holds that $\norm{\bw}_{H^{\frac{3}{2} - s}} \leq C\norm{\bw}_{H^1}$.
\end{proof}

For each $N \in \nN$, let $P_N$ denote the orthogonal projection in $H$ via $P_N : H \to H^N:=\operatorname{span}\set{\boldsymbol{\varphi}_k}_{k \leq N}$.  For $0\leq s_1<s_2<\infty$, we recall the following well-known projection estimate for any $\bu\in H^{s_2}$:
\begin{align}\label{QNproj}
 \|(I-P_N)\bu\|_{H^{s_1}}^2 
 =
 \sum_{\substack{k\in\nN\\k> N}}\lambda_k^{s_1}|\hat{\bu}_k|^{2}
 \leq
 \sum_{\substack{k\in\nN\\k> N}}\lambda_k^{s_1}\pnt{\frac{\lambda_k}{\lambda_N}}^{s_2-s_1}|\hat{\bu}_k|^{2}
 \leq
 \lambda_N^{-s_2}\|\bu\|_{H^{s_2}}^2\rightarrow0,
\end{align}
as $N\rightarrow\infty$.

 \section{Global Well-Posedness of the fractional Euler-Voigt equations}\label{sec:EV_GWP}
\noindent 
In this section we prove global well-posedness for the fractional Euler-Voigt equations, \eqref{NSV} with $\nu = 0$. 

\begin{definition}\label{def_strong_sol}
For $s\geq 0$, let $\bu^0\in H^s$.  Suppose $T>0$, and $\nu = 0$.
We say that $\bu$ is a strong solution to the fractional Euler-Voigt equations 
\eqref{fNSV}, or equivalently \eqref{fNSV_functional}, provided that $\bu\in C([0,T], V^s) \cap L^2(0,T;V^{\frac{1+s}{2}})$, and $\bu$ satisfies
\begin{equation}\label{eqn-strong}
\begin{aligned} 
       \bu(t)  + \alpha^{2s}A^{s} \bu(t) &- \bu^0 - \alpha^{2s} A^{s} \bu^0 
     & = 
    -\int_0^t  B(\bu(\tau), \bu(\tau)) \, d\tau,
\end{aligned}
\end{equation}
for all $t$,  where the  above equality is understood in the sense of $(V^s)^\prime$. Equivalently, it can be understood in the following weak sense,  i.e.

\begin{align} 
    & \quad
    (\bu(t), \bv) + \alpha^{2s}\ip{A^{s/2} \bu(t)}{A^{s/2} \bv} - (\bu^0, \bv) 
    - \alpha^{2s} \int_0^t \ip{A^{s/2} \bu^0}{A^{s/2} \bv} 
    \\ & = \notag 
    -\int_0^t \ip{B(\bu(\tau), \bu(\tau))}{\bv} \, d\tau,
\end{align}
for all $t$ and all $\bv\in V^s$ where $\ip{\cdot}{\cdot} \equiv \duality{\cdot}{\cdot}$.
\end{definition}

\begin{theorem}[Global Existence, Uniqueness] \label{thm:EV_existence_uniqueness}
Suppose $\nu = 0$, $\alpha>0$,  $r > \frac{5}{6}$, and $\bu^0\in H^r$.  Then for arbitrary $T>0$, there exists a unique solution to \eqref{fNSV_functional} in the sense of Definition \ref{def_strong_sol} with $s = r$.
\end{theorem}

\begin{proof}
Consider a Galerkin approximation to \eqref{fNSV} with $\nu = 0$ (and equivalently \eqref{fNSV_functional}) at level $N$ given by
\begin{align}
    \label{G_fEV_mo}
    (I + \alpha^{2r} A^r) \frac{d}{d t} \bu_N 
    & = 
    - P_N B(\bu_N, \bu_N), 
    \\ 
    \label{G_fEV_IC}
    \bu_N(0) 
    & = 
    P_N\bu^0.
\end{align}
Applying the operator $(I + \alpha^{2r} A^r)\inv$ to \eqref{G_fEV_mo}, we see this gives an ODE in a finite dimensional space of coefficients. Hence, existence and uniqueness of a local solution $\bu_N \in C^1([-T_N, T_N], H_N)$ for some $T_N > 0$ follows from the Picard-Lindol\"of Theorem. Now, we can take a justified inner product with $\bu_N$. After integrating by parts, we obtain
\begin{align*}
    \frac{1}{2}\frac{d}{dt} \lp \norm{\bu_N}_{L^2}^2 + \alpha^{2r}\norm{A^{r/2} \bu_N}_{L^2}^2 \rp = 0.
\end{align*}
Integrating in time then yields
\begin{align}\label{fEVG_time_space_bound}
    \norm{\bu_N(t)}_{L^2}^2 + \alpha^{2r}\norm{A^{r/2} \bu_N(t)}_{L^2}^2 
    &= 
    \norm{\bu^0_N}_{L^2}^2 + \alpha^{2r}\norm{A^{r/2} \bu_N^0}_{L^2}^2.
    \\&\leq\notag
    \norm{\bu^0}_{L^2}^2 + \alpha^{2r}\norm{A^{r/2} \bu^0}_{L^2}^2.
\end{align}
Therefore, $\norm{\bu_N}_{L^\infty(0,T, L^2)}$ is bounded uniformly in $N$. Hence, by continually restarting at time $T_N, 2T_N, 3T_N,$ etc., we may extend the time interval of existence and uniqueness arbitrarily far. Fix arbitrary $T > 0$. We will work on the interval $[0,T]$. (The same argument can be made on $[-T,0]$.) Additionally, we see from \eqref{fEVG_time_space_bound}, that the sequence $\set{\bu_N}$ is bounded in $L^\infty([0,T], V^r)$. 

As in the standard arguments for the 3D Euler equations (see, e.g., \cite{Marchioro_Pulvirenti_1994}), we will proceed to show that $\set{\bu_N}_{N\in \nN}$ is a Cauchy sequence in $C([0,T], V^r)$ for some $T>0$. Let $\bu^0 \in H^r$, and let $\bu_N$ and $\bu_M$ be solutions of \eqref{G_fEV_mo} for $N, M \in \nN$ with $N < M$. Set $\bw^N_M : = \bu_N - \bu_M \in H_M$.  
We will proceed by first handling the simpler case when $r > \frac{5}{4}$. Then we will consider the remaining parameter values, in particular $\frac{5}{6} < r \leq \frac{5}{4}$. We divide our argument into cases because Lemma \ref{lemma:bilinear_estimates} restricts $r \leq \frac{5}{4}$ due to the application of H\"older's inequality. In our arguments, we will use the following standard inequalities which follow from Parseval's identity \eqref{Hs_norm_parseval}: for $\bv \in H^s$ and $\epsilon >0$
\begin{align}\label{ineq_from_Parseval1}
    \norm{(I - P_N)\bv}_{L^2} \leq \lambda_N^{-1/2} \norm{\bv}_{H^1},
\end{align}
and more generally,
\begin{align} \label{ineq_from_Parseval}
    \norm{(I - P_N) \bv}_{H^{s-\epsilon}} \leq \lambda_N^{-\epsilon/2}\norm{\bv}_{H^s}.
\end{align}

First, assume that $ r > \frac{5}{4}$. Subtract the equations \eqref{G_fEV_mo} for $\bu_N$ and $\bu_M$ and take an inner product with $\bw_M^N$. Utilizing \eqref{switch}, Sobolev embedding, and \eqref{ineq_from_Parseval} yields
\begin{align} \label{fEVG_NM_difference1}
    & \quad \frac{1}{2} \frac{d}{dt} \lp \norm{\bw_M^N}_{L^2}^2 + \alpha^{2r} \norm{A^{r/2} \bw_M^N}_{L^2}^2 \rp 
    \\ & \notag = 
    -(B(\bw_M^N, \bu_M), \bw_M^N ) - (B(\bu_N, \bu_N), (I-P_N) \bu_M)
    \\ & \notag \leq
    C\norm{\bw_M^N}_{H^{1/2}}\norm{\bu_M}_{H^1}\norm{\bw_M^N}_{H^1} + C\norm{\bu_N}_{H^{1/2}}\norm{\bu_N}_{H^1}\norm{(I- P_N)\bu_M}_{H^1}
    \\ & \leq \notag 
    K_{\alpha, r}\norm{\bw_M^N}_{H^r}^2 + K_{\alpha, r}^3 \lambda_N^{-\epsilon/2},
\end{align}
where $\epsilon > \frac{1}{4}$. Note that $\norm{\bu_N}_{H^r}\leq K_{\alpha, r}$ uniformly in $N$. Hence, Gr\"onwall's inequality produces
\begin{align}
    &\quad \label{Cauchy}
    \norm{\bw_M^N (t)}_{L^2}^2 + \alpha^{2r} \norm{A^{r/2} \bw_M^N (t)}_{L^2}^2 
    \\&\leq \notag
    e^{C_{\alpha, r} t} \lp \norm{\bw_M^N(0)}_{L^2}^2 + \alpha^{2r}\norm{A^{r/2} \bw_M^N(0)}_{L^2}^2 \rp + e^{C_{\alpha, r} t}C_{\alpha,r}\lambda_N^{-\epsilon/2}t
    \\&\leq \notag
    e^{C_{\alpha, r} T} \lp \norm{\bw_M^N(0)}_{L^2}^2 + \alpha^{2r} \norm{A^{r/2} \bw_M^N(0)}_{L^2}^2 \rp + e^{C_{\alpha, r} T}C_{\alpha,r}\lambda_N^{-\epsilon/2}T.
\end{align}
Because $\bu^0 \in H^r$, it follows that $\norm{\bw_M^N(0)}_{L^2}^2 + \alpha^{2r} \norm{A^{r/2} \bw_M^N(0)}_{L^2}^2 \to 0$ as $M, N \to \infty$. Therefore, $\set{\bu_N}_{N\in \nN}$ is Cauchy in $C([0,T], V^r)$ and so, converges to an element $\bu \in C([0,T], V^r)$. 

Next, assume that $\frac{5}{6} < r \leq \frac{5}{4}$. As we did above, subtract the equations \eqref{G_fEV_mo} for $\bu_N$ and $\bu_M$ and take an inner product with $\bw_M^N$. Now, applying \eqref{switch} yields
\begin{align} \label{fEVG_NM_difference}
    & \quad \frac{1}{2} \frac{d}{dt} \lp \norm{\bw_M^N}_{L^2}^2 + \alpha^{2r} \norm{A^{r/2} \bw_M^N}_{L^2}^2 \rp 
    \\ & \notag = 
    -(B(\bw_M^N, \bu_M), \bw_M^N ) - (B(\bu_N, \bu_N), (I-P_N) \bu_M).
\end{align}
Using \eqref{r_52_2r_r}, \eqref{r_r_r_epsilon}, \eqref{fEVG_time_space_bound}, and \eqref{ineq_from_Parseval}, since $r > \frac{5}{6}$ implies $\frac{5}{2} - 2r < r$,  we see that the right-hand side of  \eqref{fEVG_NM_difference} is bounded above by
\begin{align}
    & \quad C\norm{\bw_M^N}_{H^r}^2 \norm{\bu_M}_{H^{\frac{5}{2} - 2r}} + C\norm{\bu_N}_{H^r} \norm{\bu_N}_{H^r} \norm{(I-P_N)\bu_M}_{H^{r-\epsilon}} 
    \\ & \leq \notag
    C\norm{\bw_M^N}_{H^r}^2 \norm{\bu_M}_{H^r} + C\norm{\bu_N}_{H^r}^2 \norm{\bu_M}_{H^r}\lambda_N^{-\epsilon/2}
    \\ & \leq \notag
    K_{\alpha, r} \norm{\bw_M^N}_{H^r}^2 + K^3_{\alpha, r}\lambda_N^{-\epsilon/2}
    \\ & \leq \notag
    K_{\alpha, r} \lp \norm{\bw_M^N}_{L^2}^2 + \alpha^2\norm{A^{r/2}\bw_M^N}_{L^2}^2 \rp + K^3_{\alpha, r}\lambda_N^{-\epsilon/2},
\end{align}
since $\norm{\bu_N}_{H^r}\leq K_{\alpha, r}$ uniformly in $N$. The rest follows as above.

Now, to show that $\bu$ satisfies \eqref{fNSV_functional} in the sense of $(V^r)'$, we choose $\bv \in \cV$ arbitrarily. As $\bu_N \in C([0,T], \cV)$, we can take a justified inner product of \eqref{G_fEV_mo} with $\bv$ and integrate in time; we obtain
\begin{align}\label{fEV_conv_eq}
    & \quad 
    (\bu_N(t), \bv) + \alpha^{2r}(A^r\bu_N(t), \bv) - (\bu_N(0), \bv) - \alpha^{2r}(A^r\bu_N(0), \bv) 
    \\ & = \notag
    -\int_0^t (B(\bu_N(s), \bu_N(s)), P_N\bv) \, ds. 
\end{align}
Since $\bu_N \to \bu$ strongly in $C([0,T], V^r)$, the sequence $(\bu_N(t), \bv) \to (\bu(t), \bv)$ for all $t \in [0,T]$. Additionally, observe
\begin{align}
    &\quad
    \int_0^t (B(\bu_N(s), \bu_N(s)), P_N\bv) \, ds - \int_0^t \ip{B(\bu(s), \bu(s))}{\bv} \, ds
    \\ & = \notag
    \int_0^t (B(\bu_N(s), \bu_N(s)), P_N\bv - \bv) \, ds + \int_0^t \ip{B(\bu_N(s) - \bu(s), \bu_N(s)))}{\bv} \, ds
    \\ & \qquad + \notag
    \int_0^t \ip{B(\bu(s), \bu_N(s) - \bu(s))}{\bv} \, ds.
\end{align}
Now by \eqref{r_52_2r_r}, we observe that 
\begin{align}
&\quad
    |(B(\bu_N(s), \bu_N(s), P_N\bv - \bv)|
    \\& \leq \notag
    C\norm{\bu_N(s)}_{H^r}\norm{\bu_N(s)}_{H^{\frac{5}{2}-2r}}\norm{P_N\bv - \bv}_{H^r}
    \\ & \leq \notag
    C\norm{\bu_N(s)}_{H^r}^2\norm{P_N\bv - \bv}_{H^r},
\end{align}
for $\frac{5}{6} \leq r \leq \frac{3}{2}$. Similarly,  
\begin{align}
   & \quad 
   |\ip{B(\bu_N(s)-\bu(s), \bu_N(s))}{\bv}|
   \\ & \leq \notag 
   C\norm{\bu_N(s) - \bu(s)}_{H^r}\norm{\bu_N(s)}_{H^r}\norm{\bv}_{H^r},
\end{align}
and 
\begin{align}
    & \quad 
    |\ip{B(\bu(s), \bu_N(s) - \bu(s))}{\bv}|
    \\ & \leq \notag 
    C\norm{\bu_N(s) - \bu(s)}_{H^r}\norm{\bu(s)}_{H^r}\norm{\bv}_{H^r}.
\end{align}
Hence, using the facts that $\norm{\bu_N}_{H^r}$ is bounded independently of $N$, that $\bu_N \to \bu$ strongly in $C([0,T], V^r)$ and that $\norm{P_N\bv - \bv}_{H^r}\to 0$ as $N \to \infty$ (by Parseval's identity), we can conclude that $\int_0^t(B(\bu_N(s), \bu_N(s)), P_N\bv) \, ds \to \int_0^t \ip{B(\bu(s), \bu(s)}{\bv} \, ds$ as $N\to \infty$. Therefore, sending $N \to \infty$ in \eqref{fEV_conv_eq}, we obtain
\begin{align} \label{fEV_existance}
&\quad
    (\bu(t), \bv) + \alpha^{2r}\ip{A^r \bu(t)}{ \bv} - (\bu(0), \bv) - \alpha^{2r}\ip{A^r\bu(0)}{\bv}
    \\ & = \notag
    -\int_0^t \ip{B(\bu(s), \bu(s))}{\bv} \, ds,
\end{align}
for all $\bv \in \cV$. Indeed, as $\cV$ is dense in $V^r$, a standard density argument shows that this equality holds for all $\bv \in V^r$. As $\bu \in C([0,T], V^r)$, it follows from \eqref{r_52_2r_r} that $A^r\bu, B(\bu, \bu) \in C([0,T], (V^r)')$. Hence, \eqref{fEV_existance} shows that \eqref{fNSV_functional} holds in the sense of $(V^r)'$ and $u(0) = u^0$. Therefore, we have proven existence of a strong solution to \eqref{fNSV_functional} in the sense of Definition \ref{def_strong_sol}. Moreover, $\frac{d}{dt}\bu \in C([0,T], V^r)$. To see this, apply $(I + \alpha^{2r} A^r)\inv$ to \eqref{fNSV_functional} and observe that $(I+\alpha^{2r} A^r)\inv B(\bu, \bu) \in C([0,T], V^r)$. Consequently, $\bu \in C^1([0,T], V^r)$. 

To prove the uniqueness and continuous dependence on initial data, consider two solutions of \eqref{fNSV_functional} $\bu_1, \bu_2 \in C([0,T], V^r)$ with initial data $\bu_1^0, \bu_2^0 \in V^r$, respectively. We denote $\widetilde{\bu} := \bu_2 - \bu_1$. Subtracting, we obtain 
\begin{align}\label{fEV_difference}
    (I + \alpha^{2r} A^r)\frac{d}{dt}\widetilde{\bu} = -B(\bu_2, \widetilde{\bu}) - B(\widetilde{\bu}, \bu_1).
\end{align}
Since $\frac{d}{dt}\widetilde{\bu} \in C([0,T], V^r)$, we can take the action of \eqref{fEV_difference} on $\widetilde{\bu}$.
Using the Lions-Magnes Lemma and \eqref{r_52_2r_r}, we obtain
\begin{align}
    \frac{1}{2}\frac{d}{dt} \lp \norm{\widetilde{\bu}}_{L^2}^2 + \alpha^{2r}\norm{A^{r/2} \widetilde{\bu}}_{L^2}^2 \rp 
    & = 
    - (B(\widetilde{\bu}, \bu_1), \widetilde{\bu})
    \\ & \leq \notag
    C\norm{\widetilde{\bu}}_{H^r}^2 \norm{\bu_1}_{H^r}
    \\ & \leq \notag
    K_{\alpha, r}\lp \norm{\widetilde{\bu}}_{L^2}^2 + \alpha^{2r}\norm{A^{r/2} \widetilde{\bu}}_{L^2}^2 \rp,
\end{align}
since $\norm{\bu_1}_{H^r} \leq K_{\alpha, r}$. By Gr\"onwall's inequality, we see that 
\begin{align}
    \norm{\widetilde{\bu}(t)}_{L^2}^2 + \alpha^{2r}\norm{A^{r/2} \widetilde{\bu}(t)}_{L^2}^2 
    \leq 
    e^{TK_{\alpha, r}}\lp \norm{\widetilde{\bu}(0)}_{L^2}^2 + \alpha^{2r}\norm{A^{r/2} \widetilde{\bu}(0)}_{L^2}^2 \rp,
\end{align}
for all $t \in [0,T]$. Hence, solutions depend continuously on initial data in the $L^\infty(0,T; V^r)$-norm. In particular, if $\bu_1^0 = \bu_2^0$, then $\bu_2(t) = \bu_1(t)$ for all $t \in [0, T]$, i.e., solutions are unique.
\end{proof}

 \section{Global Well-Posedness of the fractional Navier-Stokes-Voigt equations}\label{sec:NSV_GWP}
In this section, we prove global well-posedness for the fractional Navier-Stokes equations, \eqref{fNSV} with $\nu >0$. Although this is a similar result to that in Section \ref{sec:EV_GWP}, there are some differences. In particular, we can allow $r \geq \frac{1}{2}$ thanks to the additional smoothness provided by the viscus term. Additionally, the argument for global well-posedness uses compactness rather than a Cauchy sequence. 

Let us recall the definitions of weak and strong solutions for the Navier-Stokes-Voigt equations \eqref{fNSV}.

\begin{definition}\label{fNSV_def_strong_sol}
For $s\geq 0$, let $\bu^0\in H^s$. Suppose $T>0$ and $\nu >0$.
We say that $\bu$ is a weak solution to the fractional Navier-Stokes-Voigt equations \eqref{fNSV}, or equivalently \eqref{fNSV_functional}, on the time interval $[0,T]$ provided that $\bu\in C(0,T, V^s) \cap L^2(0,T, V \cap V^s)$, $\frac{d}{dt}\bu\in L^2(0,T, (V\cap V^s)')$, and $\bu$ satisfies
\begin{align} \label{fNSV_weak_form}
    & (\bu(t), \bv) + \alpha^{2s}(A^{s/2} \bu(t), A^{s/2}\bv) - \int_0^t\nu(A^{1/2} \bu(\tau), A^{1/2}\bv) \, d\tau 
    \\ & = \notag
    \int_0^t \ip{B(\bu(\tau), \bu(\tau))}{ \bv} \, d\tau + (\bu^0, \bv) + \alpha^{2s}(A^{s/2}\bu^0, A^{s/2}\bv),
\end{align}
for all $t$ and all $\bv\in V \cap V^s$. Furthermore, $\bu$ is a strong solution to \eqref{fNSV} if it is a weak solution and additionally, $\frac{d}{dt}\bu \in C(0,T, V \cap V^s)$.
\end{definition}

\begin{theorem}[Global Existence, Uniqueness]
Suppose $\alpha>0$, and $\bu^0\in H^r$ for $r \geq \frac{1}{2}$.  Then there exists a unique solution to \eqref{fNSV} with $\nu >0$ in the sense of Definition \ref{fNSV_def_strong_sol}.
\end{theorem}

\begin{proof}
Let $T>0$ fixed. Consider a finite dimensional Galerkin approximation of \eqref{fNSV_functional} which is based on the basis of eigenfunctions $\set{\phi_k}_{k=1}^\infty$ of the Stokes operator $A$ and is given by the following system of ODEs on $H_N : = \operatorname{span}\set{\phi_k}_{\abs{k} \leq N}$:
\begin{subequations}\label{fNSV_G}
\begin{alignat}{2}
\label{fNSV_G_mo}
(I + \alpha^{2r} A^r)\frac{d}{dt} \bu_N + \nu A\bu_N & = - P_NB(\bu_N, \bu_N),
\\ \label{fNSV_G_IC}
\bu_N(0) & = P_N\bu^0,
\end{alignat}
\end{subequations}
where we use the same notation as in Section \ref{sec:EV_GWP}. Following nearly identical steps to the proof of Theorem \ref{thm:EV_existence_uniqueness} (e.g., using Picard-Linde\"of, etc.), we obtain a Galerkin solution $\bu_N \in C^1([0, \infty), H_N)$. Integrating in time and using \eqref{switch}, we obtain 
\begin{align}\label{fNSVG_time_space_bound}
    & \quad 
    \norm{\bu_N(t)}_{L^2}^2 + \alpha^{2r}\norm{A^{r/2}\bu_N(t)}_{L^2}^2 + 2\nu \int_0^t \norm{\nabla \bu_N(s)}_{L^2}^2 \, ds
    \\ & =  \notag
    \norm{\bu_N^0}_{L^2}^2 + \alpha^{2r}\norm{A^{r/2}\bu_N^0}_{L^2}^2.
\end{align}
Furthermore, we see from \eqref{fNSVG_time_space_bound} that for fixed, arbitrary $T > 0$, $\bu_N$ is bounded in $L^\infty([0,T], V^r) \cap L^2([0,T], V)$, uniformly independent of $N$. 

We aim to extract a subsequence of $\set{\bu_N}$ which converges strongly in $L^2((0,T), V)$ by using the Aubin Compactness Theorem. With this in mind, we must prove $\frac{d}{dt}\bu_N$ is uniformly bounded in $L^2((0,T), V')$ with respect to $N$: Applying $(I+\alpha^{2r} A^r)\inv$ to equation \eqref{fNSV_G_mo}, we obtain
\begin{align}
    \frac{d}{dt}\bu_N + \nu (I+\alpha^{2r} A^r)\inv A \bu_N = -(I+\alpha^{2r} A^r)\inv P_N B(\bu_N, \bu_N).
\end{align}
Taking the action of $\bv \in V$ with the above equality gives 
\begin{align}
    \ip{\frac{d}{dt}\bu_N}{\bv} = -\nu \ip{(I+\alpha^{2r} A^r)\inv  A\bu_N}{\bv} - \ip{(I+\alpha^{2r} A^r)\inv P_NB(\bu_N, \bu_N)}{\bv}.
\end{align}
We estimate
\begin{align}
    \abs{\ip{(I+\alpha^{2r} A^r)\inv \nu A\bu_N}{\bv}} 
    & = 
    \nu \abs{\ip{(I+\alpha^{2r} A^r)\inv A^{1/2} \bu_N}{A^{1/2}\bv}}
    \\ & \leq \notag
    \nu \norm{(I+\alpha^{2r} A^r)\inv A^{1/2} \bu_N}_{L^2} \norm{A^{1/2}\bv}_{L^2}
    \\ & \leq \notag
    \nu \norm{(I+\alpha^{2r} A^r)\inv A^{1/2} \bu_N}_{L^2} \norm{\bv}_V
    \\ & \leq \notag 
    C \nu \norm{\bu_N}_{H^{1-2r}}\norm{\bv}_V
    \\ & \leq \notag 
    C \nu \norm{\bu_N}_{H^1} \norm{\bv}_V.
\end{align}
Now, using \eqref{switch}, \eqref{r_52_2r_r}, and the fact that $(I+\alpha^{2r} A^r)\inv$ is self-adjoint, we obtain
\begin{align}
    \abs{\ip{(I+\alpha^{2r} A^r)\inv P_N B(\bu_N, \bu_N)}{\bv}}
    & = 
    \abs{\ip{P_N B(\bu_N, \bu_N)}{(I+\alpha^{2r} A^r)\inv \bv}}
    \\ & = \notag
    \abs{\ip{P_NB(\bu_N, (I+\alpha^{2r} A^r)\inv\bv)}{\bu_N}}
    \\ & = \notag 
    \abs{\ip{B(\bu_N, (I+\alpha^{2r} A^r)\inv\bv)}{\bu_N}}  
    \\ & \leq \notag
    C_{\alpha, r} \norm{\bu_N}^2_{H^r} \norm{(I+\alpha^{2r} A^r)\inv \bv}_{H^{\frac{5}{2} - 2r}}
    \\ & \leq \notag
    C_{\alpha, r} \norm{\bu_N}^2_{H^r} \norm{ \bv}_{H^{\frac{5}{2} - 4r}}
    \\ & \leq \notag
    C_{\alpha, r} \norm{\bu_N}^2_{H^r} \norm{ \bv}_{H^{1}}.
\end{align}
We note in passing that the last inequality holds for $\frac{3}{8} \leq r \leq \frac{3}{2}$ rather than just $\frac{5}{6} \leq r \leq \frac{3}{2}$ (as stated in \eqref{r_52_2r_r}) since $\bv \in V$. Due to these estimates, we have shown 
\begin{align}
    \abs{\ip{\frac{d}{dt} \bu_N}{\bv}} 
    \leq 
    \nu \norm{\bu_N}_{H^1} \norm{\bv}_V + C_{\alpha, r} \norm{\bu_N}^2_{H^r} \norm{\bv}_{H^1}.
\end{align}
Squaring and integrating in time, we obtain the following
\begin{align*}
    \norm{\frac{d}{dt}\bu_N}_{L^2(0,T; V')} 
    & \leq 
    C \nu \int_0^T \norm{\bu_N}_{H^1}^2 \, dt + C_{\alpha, r} \int_0^T \norm{\bu_N}_{H^r}^4 \, dt 
    \\ & \leq \notag 
    C \nu \norm{\bu_N}_{L^2(V)}^2 + T C_{\alpha, r}\norm{\bu_N}^4_{L^\infty (V^r)}.
\end{align*}
Hence, owing to the above bounds, we have shown the sequence   
\begin{align}\label{fNSVG_time_derivative_bound}
    \set{\frac{d}{dt}\bu_N} \text{ is bounded in } L^2([0,T], V').
\end{align}
Thus, by the Aubin Compactness Theorem (see, e.g., \cite{Constantin_Foias_1988}), there exists a subsequence of $\set{\bu_N}$ (which we relabel as $\set{\bu_N}$ if necessary) and an element $\bu \in C([0,T], H)$ such that 
\begin{align}\label{strong_conv}
    \bu_N \to \bu \ \text{ strongly in } \ L^\infty([0,T], H).
\end{align}
Further, using the Banach-Alaoglu Theorem and the uniform boundedness of $\bu_N$ in $L^\infty([0,T], V^r) \cap L^2([0,T], V)$ we can pass to an additional subsequence if necessary (which again we relabel as $\set{\bu_N}$), to show that, in fact, $\bu \in L^\infty([0,T], V^r) \cap L^2([0,T], V)$, and 
\begin{align}\label{weak_conv}
    &\bu_N \rightharpoonup \bu \ \text{ weakly in } \ L^2([0,T], V),
    \\ & \label{weak_star_conv}
    \bu_N \rightharpoonup \bu \ \text{ weak-$\ast$ in } \ L^\infty([0,T], V^r),
    \\ & \label{time_deriv_conv}
    \frac{d}{dt}\bu_N \rightharpoonup \frac{d}{dt}\bu \text{ weakly in } L^2([0,T], V).
\end{align}
One should note that by passing to additional subsequences,
the limiting objects above coincide; in particular, \eqref{time_deriv_conv} follows by a standard argument involving integrating by parts in time.

Now, we will show $\bu$ satisfies \eqref{fNSV} in the sense of $L^2(0,T; V')$. We choose an arbitrary element $\bv \in \cV$ and \dela{take an inner product with \eqref{fNSV_G_mo}, then integrate in time over $[0,t]$, for some $t \in [0,T]$,} and apply   
\cite[Lemma III.1.2]{Temam_2001_Th_Num} to obtain
\begin{align}\label{fNSV_converge}
    & (\bu_N(t), \bv) + \alpha^{2r}(A^r \bu_N(t), \bv) - (\bu_N(0), \bv) - \alpha^{2r}(A^r\bu_N(0), \bv) - \int_0^t\nu (A^{1/2}\bu_N(s), A^{1/2}\bv) \, ds 
    \\ & = \notag
    \int_0^t (P_N B(\bu_N(s), \bu_N(s)), \bv) \, ds.
\end{align}
We will now show that each of the terms in the above expression converges to the appropriate limit. 
First, thanks to \eqref{weak_conv}
\begin{align}
    -\nu \int_0^t (A^{1/2}\bu_N(s), A^{1/2}\bv) \, ds \to -\nu\int_0^t (A^{1/2}\bu(s), A^{1/2}\bv) \, ds.
\end{align}
Next, we will show the convergence of the bilinear term. Observe that 
\begin{align}
    &\quad
    \int_0^t (B(\bu_N(s), \bu_N(s)), P_N\bv) \, ds - \int_0^t \ip{B(\bu(s), \bu(s))}{\bv} \, ds
    \\ & = \notag
    \int_0^t (B(\bu_N(s), \bu_N(s)), P_N\bv - \bv) \, ds + \int_0^t \ip{B(\bu_N(s) - \bu(s), \bu_N(s))}{\bv} \, ds
    \\ & \qquad + \notag
    \int_0^t \ip{B(\bu(s), \bu_N(s) - \bu(s))}{\bv} \, ds
    \\ & =: \notag 
    I_1(N) + I_2(N) + I_3(N).
\end{align}
Using the facts that $\norm{\bu_N}_{L^2}$ is bounded independently of $N$, that $\norm{P_N\bv - \bv}_{H^1} \to 0$ as $N \to \infty$ by  \eqref{QNproj}, and that \eqref{strong_conv} holds, standard arguments as in the theory of the Navier-Stokes equations yield $I_1(N) \to 0$ and $I_2(N) \to 0$ as $N \to \infty$. For $I_3(N)$, first note that for all $\bu, \bv, \bw \in V$, by H\"older's inequality 
\begin{align}
    \ip{\bu \otimes \bv}{\bw}_{V'} \leq C\norm{\bu}_{L^2} \norm{\bv}_{L^3} \norm{\bw}_{L^6} \leq C \norm{\bu}_{L^2} \norm{\bv}_{H^{\frac{1}{2}}} \norm{\bw}_{H^1}.
\end{align}
Hence, $\bu \otimes \bv \in L^2(0,T; V')$. From this, and \eqref{weak_conv}, it follows that $I_3(N) \to 0$ as $N \to \infty$.
Hence, sending $N \to \infty$ in \eqref{fNSV_G} yields \eqref{fNSV_def_strong_sol} for all $\bv \in \cV$. A standard argument using the density of $\cV$ in $V\cap V^r$ shows that the equality holds for all $\bv \in V\cap V^r$. 
Hence, we have proven global existence of weak solutions according to Definition \ref{fNSV_def_strong_sol}.

We must now show the uniqueness and continuous dependence on initial data. Let $\bu_1, \bu_2 \in C([0,T], V^r) \cap L^2(0,T; V \cap V^r)$ be two solutions of \eqref{fNSV_functional} with initial data $\bu_1^0, \bu_2^0 \in V^r$ respectively. Set $\widetilde{\bu} = \bu_2 - \bu_1$. Subtracting, we obtain 

\begin{align}
    (I + \alpha^{2r} A^r)\frac{d}{dt}\widetilde{\bu} + \nu A\widetilde{\bu} = -B(\bu_1, \widetilde{\bu}) - B(\widetilde{\bu}, \bu_2).
\end{align} 
Applying the Lions–Magenes lemma, i.e., \cite[Lemma III.1.2]{Temam_2001_Th_Num} and using \eqref{r_2_32_r} then yields  
\begin{align}\label{fNSV_unique}
    & \quad \frac{1}{2}\frac{d}{dt} \lp \norm{\widetilde{\bu}}_{L^2}^2 + \alpha^{2r}\norm{A^{r/2}\widetilde{\bu}}_{L^2}^2 \rp + \nu \norm{A^{1/2} \widetilde{\bu}}_{L^2}^2 
    \\ & = \notag 
    (B(\widetilde{\bu}, \widetilde{\bu}), \bu_2)
    \\ & \leq \notag
    C \norm{\widetilde{\bu}}_{H^r} \norm{A^{1/2}\widetilde{\bu}}_{L^2} \norm{\bu_2}_{H^{\frac{3}{2}-r}}
    \\ & \leq \notag
    C_\nu \norm{\widetilde{\bu}}_{H^r}^2\norm{\bu_2}_{H^{\frac{3}{2} -r}}^2 + \frac{\nu}{2}\norm{A^{1/2}\widetilde{\bu}}_{L^2}^2
    \\ & \leq \notag 
    C_\nu \norm{\widetilde{\bu}}_{H^r}^2\norm{\bu_2}_{H^1}^2 + \frac{\nu}{2}\norm{A^{1/2}\widetilde{\bu}}_{L^2}^2,
\end{align}
thanks to \eqref{switch} and \eqref{r_2_32_r}. Applying Gr\"onwall's inequality, we obtain 
\begin{align}\label{fNSV_unique2}
    \norm{\widetilde{\bu}(t)}_{L^2}^2 + \alpha^{2r}\norm{A^{r/2} \widetilde{\bu}(t)}_{L^2}^2 
    \leq 
    e^{TC_{\alpha, r,\nu}}\lp \norm{\widetilde{\bu}(0)}_{L^2}^2 + \alpha^{2r}\norm{A^{1/2} \widetilde{\bu}(0)}_{L^2}^2 \rp,
\end{align}
for all $t \in [0,T]$. Hence, it follows that the solution depends continuously on the initial data in the $L^\infty(0, T; V^r)$-norm,
and by integrating \eqref{fNSV_unique2}, the solution depends continuously on initial data in the $L^2(0,T;V)$-norm. Additionally, if $\bu_1^0 = \bu_2^0$, then $\bu_2(t) = \bu_1(t)$ for all $t \in [0, T]$. 
\end{proof}

 \section{Convergence in the limiting case}\label{sec:Convergence}

Consider the functional formulation of the unforced Navier-Stokes equations
\begin{subequations}\label{NSE_projected}
\begin{alignat}{2}
\label{NSE_projected_mo}
\partial_t \bu + B(\bu, \bu) + \nu A \bu & = \mathbf{0},
\\ \label{NSE_projected_IC}
\bu(0) & = \bu^0.
\end{alignat}
\end{subequations}
and let $\bu_{\alpha, r}$ be a solution of the fractional Navier-Stokes-Voigt equations \eqref{fNSV_functional} with $\bu_{\alpha, r}(0) = \bu_{\alpha, r}^0$.

Let us recall the following definitions and results about the  3D Euler equations, namely \eqref{NSE_projected} with $\nu = 0$ and periodic boundary conditions, based mainly on the exposition from paper \cite{Beale_Kato_Majda_1984} and chapter 2 of \cite{Temam_1976}. 
\begin{definition}\label{def-Euler-local solution}
Assume that $s> \frac{5}{2}$ and $\bu^0 \in H^s$. A local solution to \eqref{NSE_projected} with $\nu = 0$ is function 
\begin{equation}
\label{eqn-Euler-local solution} 
\bu \in C([0,T),H^s) \cap  C^1([0,T);H^{s-1}),
\end{equation}
for some $T \in [0,\infty]:=[0,\infty)\cup\{+\infty\}$, which solves     \eqref{NSE_projected} with $\nu = 0$ in the weak sense.

A local solution $\bu$ satisfying the condition \eqref{eqn-Euler-local solution}  with $T=\infty$ is called a global solution.
A local solution $\bu$ satisfying the condition \eqref{eqn-Euler-local solution} with some 
$T \in [0,\infty]$ is called a local maximal solution if there exists no a local solution 
$\bv \in C([0,\tau),H^s) \cap C^1([0,\tau);H^{s-1})$ with $\tau>T$ such that 
the restriction of $\bv$ to the interval $[0,T)$ is equal to $\bu$.
\end{definition}
\begin{theorem}
\label{thm-Euler-local solution}
Assume that $s>5/2$ and $u_0 \in H^s$ then there exists  a local solution to \eqref{NSE_projected} with $\nu = 0$. 
Moreover, if $R_0>0$, then there exists $T_0=T_0(R_0)>0$ such that 
if $s>5/2$ and  $\bu^0 \in H^s$ with $\norm{ \bu^0 }_{H^s} \leq R_0$, then there exists 
a local solution $\bu \in C([0,T],H^s)  \in C^1([0,T];H^{s-1})$ with $T\geq T_0$. 
\end{theorem}

\noindent For a proof, see \cite{Temam_1976}.  See also \cite{Marsden+Ebin+Fischer_1972,Beale_Kato_Majda_1984,Kozono_Ogawa_Taniuchi_2002,Kozono_Taniuchi_2000}, and the references therein.

We also recall the following definitions of weak and strong solutions of the 3D Navier-Stokes equations (see, e.g., \cite{Constantin_Foias_1988,Temam_2001_Th_Num}). 

\begin{definition}
A weak solution of the 3D Navier-Stokes equations, namely \eqref{NSE_projected} with $\nu >0 $, is a function $\bu \in L^2((0,T), V) \cap C_w([0,T], H)$ satisfying $\frac{d\bu}{dt} \in L^{4/3}_{\operatorname{loc}}((0,T), V')$ and 
\begin{align}
    & \ip{\frac{d\bu}{dt}}{\bv} + \nu(A^{1/2} \bu, A^{1/2} \bv) + (B(\bu, \bu), \bv) = \mathbf{0}, 
    \\ & \bu(0) = \bu^0,
\end{align}
almost everywhere in $t$, for all $\bv \in V$. The space $C_w((0,T), H)$ is all weakly continuous functions in $L^\infty((0,T), H)$. A function $\bu$ is weakly continuous if $(\bu(t), \bv)$ is a continuous function for all $\bv \in H$. 

A strong solution of the 3D Navier-Stokes equations \eqref{NSE_projected} is a weak solution given $\bu^0 \in V$, satisfying $\bu \in C([0,T], V) \cap L^2([0,T], \cD(A))$ and $\frac{d}{dt}\bu \in L^{2}((0,T); H)$ (see, e.g., \cite{Temam_1995_Fun_Anal}). 
\end{definition}

Furthermore, we cite the following result (see, e.g., \cite{robinson_Rodrigo_Sadowski_3DNSE_2016}).

\begin{theorem}
If $\bu$ is a strong solution of the Navier-Stokes equations on $[0,T]$ with $\bu^0 \in V\cap H^m$, then for $m \in \nN \cup \set{0}$,
\begin{align}
    \bu \in L^\infty((0,T), H^m) \cap L^2((0,T), H^{m+1}).
\end{align}
\end{theorem}
We will show that the solutions to the fractional Euler-Voigt equations (fEV), i.e. \eqref{fNSV_functional} with $\nu = 0$ converge to the Euler equations (E), i.e., \eqref{NSE_projected} with $\nu = 0$. Next, we will show that solutions of the fractional Navier-Stokes-Voigt equations (NSV) converge to solutions of the Navier Stokes equations (NS). Finally, we will see convergence of solutions of fractional Navier-Stokes-Voigt (fNSV) to solutions of Euler. A visualization of this convergence is given in Figure \ref{fig:limiting_diag}; we show the dashed arrows (horizontal and diagonal), and the left solid, vertical arrow.
The convergence represented by right solid, vertical arrow has been proved when the domain is $\nR^n$ by several authors under appropriate assumptions on the initial data, see \cite{Kato_1972} and the references therein. However, in a domain with a boundary, this remains a challenging open problem.

\begin{figure}[H]
    \centering
    \begin{tikzcd}[arrows={-stealth}]
  \text{fNSV} \rar[dashed, red, "\alpha \to 0"] \dar[swap, blue, "\nu \to 0" ]\drar[dashed, red, shift left]\drar[blue, shift right]%
  & \text{NS} \dar[blue, "\nu \to 0"] \\
  \text{fEV} \rar[dashed, swap, red, "\alpha \to 0"]
  & \text{E}
\end{tikzcd}
\caption{Diagram of possible limiting for solutions to fractional Navier-Stokes-Voigt (fNSV), fractional Euler-Voigt (fEV), Navier-Stokes (NS), and Euler (E)}
\label{fig:limiting_diag}
\end{figure}
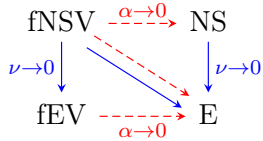

\begin{theorem}[Convergence fEV $\to$ E]\label{convergence_fEV_E}
Let $\nu = 0$, $s > 5/2$, and $r \in \left( \frac{5}{6}, 1\right]$. Given initial data $\bu^0, \bu_{\alpha, r}^0\in H^s$,

let $\bu, \bu_{\alpha, r} \in C([0,T], H^s) \cap C^1([0,T], H^{s-1})$ be the solutions corresponding to \eqref{NSE_projected}  and \eqref{fNSV_functional}, respectively on $[0,T]$ where $0 < T < T_{\operatorname{max}}$, and $T_{\operatorname{max}}$ is the maximal time for which a solution to the Euler equations exists and is unique. Then
\begin{enumerate}[(i)]
    \item There exists $C>0$ for every $t \in [0,T]$ such that
    \begin{align}\label{convergence_estimate}
        & \quad \norm{(\bu - \bu_{\alpha, r})(t)}_{L^2}^2 + \alpha^{2r}\norm{A^{r/2}(\bu - \bu_{\alpha, r})(t)}_{L^2}^2
        \\ & \leq \notag 
        \lp \norm{\bu^0 - \bu_{\alpha, r}^0}_{L^2}^2 
        + 
        \alpha^{2r} \norm{A^{r/2}(\bu^0 - \bu_{\alpha, r}^0)}_{L^2}^2 \rp 
        e^{C_Tt} 
        + 
        CT\alpha^{2r}\norm{A^{r/2}\partial_t \bu}_{L^2(0,t;L^2)}^2 e^{C_Tt}.
   \end{align}
   where $C_T = \frac{C}{T}$.
   \item Consequently, if 
   \begin{align}\label{ur_to_u_fEV}
       \norm{\bu^0 - \bu_{\alpha, r}^0}_{L^2}^2 + \alpha^{2r}\norm{A^{r/2} (\bu^0 - \bu_{\alpha, r}^0)}_{L^2}^2 \to 0 \text{ as } \alpha \to 0
   \end{align}
   (in particular, when $\bu_{\alpha, r}^0 = \bu^0$), then $\bu_{\alpha, r} \to \bu$ in $C([0,T], H^r)$ as $\alpha\to 0$. 
\end{enumerate}
\end{theorem}

\begin{proof}
Let us choose and fix $T \in (0, T_{\ast})$ Subtracting \eqref{fNSV_functional} from \eqref{NSE_projected} yields 
\begin{align}\label{convergence_difference}
    \quad & \partial_t(\bu - \bu_{\alpha, r}) - \alpha^{2r}A^r\partial_t \bu_{\alpha, r} 
    \\ & = \notag
    -B(\bu, \bu) + B(\bu_{\alpha, r}, \bu_{\alpha, r})
    \\ & = \notag
    -B(\bu - \bu_{\alpha, r}, \bu) + B(\bu - \bu_{\alpha, r}, \bu - \bu_{\alpha, r}) - B(\bu, \bu - \bu_{\alpha, r}).
\end{align}
Observe that since $\bu, \bu_{\alpha, r} \in C^1([0,T], H^{s-1})$ for $s \geq 3$, it holds that
\begin{align}\label{convergence_spaces}
    \partial_t \bu, \partial_t\bu_{\alpha, r}, A^r\partial_t\bu_{\alpha, r} \in C([0,T], L^2).
\end{align}
Additionally, as $H^{s-1}$ is an algebra for $s \geq \frac{5}{2}$, the bilinear terms are also in $C([0,T], H^{s-1})$. Hence, we may take a justified inner product of \eqref{convergence_difference} with $\bu - \bu_{\alpha, r}$. Using 
 \eqref{switch}, one finds that 
\begin{align}
    & \quad (\partial_t(\bu - \bu_{\alpha, r}), \bu - \bu_{\alpha, r}) - \alpha^{2r}(A^r\partial_t \bu_{\alpha, r}, \bu- \bu_{\alpha, r}) 
    \\ & = \notag
    -(B(\bu - \bu_{\alpha, r}, \bu), \bu - \bu_{\alpha, r}).
\end{align}
Adding and subtracting $\alpha^{2r}(A^r \partial_t \bu, \bu - \bu_{\alpha, r})$ using the Lions-Magenes Lemma (see \cite[Lemma III.1.2]{Temam_2001_Th_Num})
and then integrating by parts yields
\begin{align}\label{convergence_IBP}
    & \quad \frac{1}{2} \frac{d}{dt}\lp \norm{\bu - \bu_{\alpha, r}}_{L^2}^2 + \alpha^{2r} \norm{A^{r/2}(\bu - \bu_{\alpha, r})}_{L^2}^2 \rp  
    \\ & = \notag
    - \alpha^{2r}(\partial_t \bu, A^r(\bu - \bu_{\alpha, r}) - (B(\bu - \bu_{\alpha, r}, \bu), \bu - \bu_{\alpha, r}).
\end{align}
Now using \eqref{convergence_spaces}, we obtain the following estimate 
\begin{align}
    \alpha^{2r}\abs{(A^{r/2}\partial_t \bu, A^{r/2}(\bu - \bu_{\alpha, r}))}
    & \leq
    \alpha^{2r}\norm{A^{r/2}\partial_t \bu}_{L^2} \norm{A^{r/2}(\bu - \bu_{\alpha, r})}_{L^2}
    \\ & = \notag
    \lp T^{1/2}\alpha^{r}\norm{A^{r/2}\partial_t \bu}_{L^2} \rp \lp T^{-1/2} \alpha^r \norm{A^{r/2}(\bu - \bu_{\alpha, r})}_{L^2} \rp 
    \\ & \leq \notag 
    \frac{T}{2} \alpha^{2r} \norm{A^{r/2}\partial_t \bu}_{L^2}^2 
    + 
    \frac{1}{2T} \alpha^{2r} \norm{A^{r/2}(\bu - \bu_{\alpha,r})}_{L^2}^2.
\end{align}
Because $\bu$ is a regular solution (i.e. $\bu \in H^s$ with $s > 5/2$), we obtain 
\begin{align}\label{fEV_E_bilinear}\notag 
    \abs{(B(\bu - \bu_{\alpha, r}, \bu), \bu - \bu_{\alpha, r})} 
    & \leq 
    \norm{\bu - \bu_{\alpha, r}}_{L^2}^2 \norm{\nabla \bu}_{L^\infty} 
    \\ & \leq \notag
    C\norm{\bu - \bu_{\alpha, r}}_{L^2}^2 \norm{\bu}_{H^s} 
    \\ & \leq \notag 
    \Tilde{C}\norm{\bu - \bu_{\alpha, r}}_{L^2}^2.
\end{align}

Collecting the estimates, \eqref{convergence_IBP} can be bounded above as follows 
\begin{align}
    & \quad \frac{1}{2} \frac{d}{dt}\lp \norm{\bu(t) - \bu_{\alpha, r}(t)}_{L^2}^2 + \alpha^{2r} \norm{A^{r/2}(\bu(t) - \bu_{\alpha, r}(t))}_{L^2}^2 \rp 
    \\ & \leq \notag 
    \frac{C}{2T} \lp \norm{\bu - \bu_{\alpha, r}}_{L^2}^2 
    + 
    \alpha^{2r}\norm{A^{r/2}(\bu - \bu_{\alpha, r})}_{L^2}^2 \rp 
    + 
    \frac{T}{2} \alpha^{2r} \norm{A^{r/2}\partial_t \bu}_{L^2}^2 .
\end{align}
Gr\"onwall's inequality then yields 
\begin{align}\label{limit}
    & \quad \norm{(\bu - \bu_{\alpha, r})(t)}_{L^2}^2 
    + 
    \alpha^{2r} \norm{A^{r/2}(\bu - \bu_{\alpha, r})(t)}_{L^2}^2 
    \\ &\leq \notag 
    \lp \norm{\bu^0 - \bu_{\alpha, r}^0}_{L^2}^2 
    + 
    \alpha^{2r} \norm{A^{r/2}(\bu^0 - \bu_{\alpha, r}^0)}_{L^2}^2 \rp 
    e^{C_Tt} 
    + 
    CT\alpha^{2r}\norm{A^{r/2}\partial_t \bu}_{L^2(0,t;L^2)}^2 e^{C_Tt},
\end{align}
where $C_T = \frac{C}{T}$.
Hence, \eqref{convergence_estimate} is proved. Finally, if $\lim\limits_{\alpha\to 0}\norm{\bu^0 - \bu^0_{\alpha, r}}_{L^2} = 0$ and $\norm{A^{r/2}\bu_{\alpha, r}^0}_{L^2}^2 \leq M < \infty$ for all $\alpha$, taking the $\limsup$ of \eqref{limit} as $\alpha \to 0$ shows that $\bu_{\alpha, r} \to \bu$ in $C([0,T), H^r)$. 
\end{proof}

\begin{theorem}[Convergence fNSV $\to$ NS]\label{fNSV_to_NS}
Let $s > 1$, $\nu >0$, and $r \in \left[ \frac{1}{2}, 1 \right]$. Given initial data $\bu^0, \bu_{\alpha, r}^0\in V\cap H^{r+s}$, let $\bu, \bu_{\alpha, r} \in C([0,T], H^{r+s})\cap L^2((0,T), H^{r+s+1})$ be the solutions corresponding to \eqref{NSE_projected} and \eqref{fNSV_functional}, respectively on $[0,T]$ where $0 < T < T_{\operatorname{max}}$, and $T_{\operatorname{max}}$ is the maximal time for which a strong solution to the Navier-Stokes equations exists and is unique. Then: 
\begin{enumerate}[(i)]
    \item There exists $C>0$ for all $t \in [0,T]$ such that 
    \begin{align}\label{convergence_estimate_fNSV_NSE}
        & \quad \norm{(\bu - \bu_{\alpha, r})(t)}_{L^2}^2 + \alpha^{2r}\norm{A^{r/2}(\bu - \bu_{\alpha, r})(t)}_{L^2}^2
        \\ & \leq \notag 
        \lp \norm{\bu^0 - \bu_{\alpha, r}^0}_{L^2}^2 
        + 
        \alpha^{2r} \norm{A^{r/2}(\bu^0 - \bu_{\alpha, r}^0)}_{L^2}^2 \rp 
        e^{C_Tt} 
        + 
        CT\alpha^{2r}\norm{A^{r/2}\partial_t \bu}_{L^2(0,t;L^2)}^2 e^{C_Tt},
   \end{align}
   where $C_T = \frac{C}{T}$.

   \item Consequently, if 
   \begin{align}\label{ur_to_u_fNSV}
       \norm{\bu^0 - \bu_{\alpha, r}^0}_{L^2}^2 + \alpha^{2r}\norm{A^{r/2} (\bu^0 - \bu_{\alpha, r}^0)}_{L^2}^2 \to 0 \text{ as } \alpha \to 0
   \end{align}
   (in particular, when $\bu_{\alpha, r}^0 = \bu^0$), then $\bu_{\alpha, r} \to \bu$ in $L^2((0,T), H^1)$ as $\alpha\to 0$. 
\end{enumerate}
\end{theorem}

\begin{remark}
In the following proof, we take advantage of the fact that $H^{r+s}$ is an algebra for $r \geq \frac{1}{2}$ and $s>1$ in order to control the bilinear term. 
\end{remark}

For the sake of completeness, we give the full proof of Theorem \ref{fNSV_to_NS} below. However, we note that the proof is very similar to that of Theorem \ref{convergence_fEV_E} with minor difference in the estimates. 

\begin{proof}
Subtracting \eqref{fNSV_functional} from \eqref{NSE_projected} gives 
\begin{align}\label{convergence_difference_fNSV_NSE}
    & \quad \partial_t(\bu - \bu_{\alpha, r}) - \alpha^{2r}A^r\partial_t \bu_{\alpha, r} - \nu A(\bu - \bu_{\alpha, r})
    \\ & = \notag
    -B(\bu, \bu) + B(\bu_{\alpha, r}, \bu_{\alpha, r})
    \\ & = \notag
    -B(\bu - \bu_{\alpha, r}, \bu) + B(\bu - \bu_{\alpha, r}, \bu - \bu_{\alpha, r}) - B(\bu, \bu - \bu_{\alpha, r}).
\end{align}
Let $\bu-\bu_{\alpha, r}$ act on \eqref{convergence_difference_fNSV_NSE}:
\begin{align}
    & \quad ( \partial_t(\bu - \bu_{\alpha, r}), \bu - \bu_{\alpha, r} ) - \alpha^{2r}\langle A^r\partial_t \bu_{\alpha, r}, \bu- \bu_{\alpha, r}\rangle - \nu (A(\bu - \bu_{\alpha, r}), \bu - \bu_{\alpha, r})
    \\ & = \notag
    - ( B(\bu - \bu_{\alpha, r}, \bu), \bu - \bu_{\alpha, r} ).
\end{align}
Adding and subtracting $\alpha^{2r}( A^r \partial_t \bu, \bu - \bu_{\alpha, r} )$, then integrating by parts yields 
\begin{align}\label{convergence_IBP_fNSV_NSE}
    & \quad \frac{1}{2} \frac{d}{dt}\lp \norm{\bu - \bu_{\alpha, r}}_{L^2}^2 + \alpha^{2r} \norm{A^{r/2}(\bu - \bu_{\alpha, r})}_{L^2}^2 \rp + \nu \norm{A^{1/2}(\bu - \bu_{\alpha, r})}_{L^2}^2 
    \\ & = \notag
    - \alpha^{2r} ( A^{r/2}\partial_t \bu, A^{r/2}(\bu - \bu_{\alpha, r}) ) 
    - ( B(\bu - \bu_{\alpha, r}, \bu), \bu - \bu_{\alpha, r} ).
\end{align}
Since $\bu, \bu_{\alpha, r} \in L^\infty((0,T), H^{r+s}) \cap L^2((0,T), H^{r+s+1})$ for $r \geq \frac{1}{2}$ and $s > 1$, it follows that $A\bu \in L^2((0,T), H^r)$ and $\nabla \bu \in L^\infty((0,T), H^r) \cap L^2((0,T), H^{r+1})$. Then because $H^{r+s}$ is an algebra,  $B(\bu, \bu) \in L^2((0,T); H^{r+s})$, and thus, $\partial_t \bu \in L^2((0,T), H^r)$. Using these facts, we obtain the following estimates:
\begin{align}\label{fNSV_to_NS_bilinear_bound}
    \abs{(B(\bu - \bu_{\alpha, r}, \bu), \bu - \bu_{\alpha, r})} 
    & \leq 
    \norm{\bu - \bu_{\alpha, r}}_{L^2}^2 \norm{\nabla \bu}_{L^\infty} 
    \\ & \leq \notag 
    C\norm{\bu - \bu_{\alpha, r}}_{L^2}^2 \norm{\bu}_{H^{r+2}} 
    \\ & \leq \notag 
    \Tilde{C}\norm{\bu - \bu_{\alpha, r}}_{L^2}^2,
\end{align}
and for some $T>0$ (which we utilize to obtain dimensionally correct estimates)
\begin{align}\label{fNSV_to_NS_time_derivative_term_bound}
    \alpha^{2r}\abs{(A^{r/2}\partial_t \bu, A^{r/2}(\bu - \bu_{\alpha, r}))}
    & \leq
    \alpha^{2r}\norm{A^{r/2}\partial_t \bu}_{L^2} \norm{A^{r/2}(\bu - \bu_{\alpha, r})}_{L^2}
    \\ & = \notag
    \lp T^{1/2}\alpha^{r}\norm{A^{r/2}\partial_t \bu}_{L^2} \rp \lp T^{-1/2} \alpha^r \norm{A^{r/2}(\bu - \bu_{\alpha, r})}_{L^2} \rp 
    \\ & \leq \notag 
    \frac{T}{2} \alpha^{2r} \norm{A^{r/2}\partial_t \bu}_{L^2}^2 
    + 
    \frac{1}{2T} \alpha^{2r} \norm{A^{r/2}(\bu - \bu_{\alpha,r})}_{L^2}^2.
\end{align}
Collecting these estimates, \eqref{convergence_IBP_fNSV_NSE} can be bounded above as follows 
\begin{align}\notag 
    & \quad \frac{1}{2} \frac{d}{dt}\lp \norm{\bu - \bu_{\alpha, r}}_{L^2}^2 + \alpha^{2r} \norm{A^{r/2}(\bu - \bu_{\alpha, r})}_{L^2}^2 \rp + \nu \norm{A^{1/2}(\bu - \bu_{\alpha, r})}_{L^2}^2 
    \\ & \leq \notag 
    \frac{C}{2T} \lp \norm{\bu - \bu_{\alpha, r}}_{L^2}^2 
    + 
    \alpha^{2r}\norm{A^{r/2}(\bu - \bu_{\alpha, r})}_{L^2}^2 \rp 
    + 
    \frac{T}{2} \alpha^{2r} \norm{A^{r/2}\partial_t \bu}_{L^2}^2. 
\end{align}
Gr\"onwall's inequality then yields 
\begin{align*}
    & \quad \norm{\bu(t) - \bu_{\alpha, r}(t)}_{L^2}^2 
    + 
    \alpha^{2r} \norm{A^{r/2}(\bu(t) - \bu_{\alpha, r}(t))}_{L^2}^2 
    + 
    2\nu\int_0^t e^{C_T(t-s)}\norm{A^{1/2}(\bu - \bu_{\alpha, r})(s)}_{L^2}^2 \, ds 
    \\ &\leq \notag 
    \lp \norm{\bu^0 - \bu_{\alpha, r}^0}_{L^2}^2 
    + 
    \alpha^{2r} \norm{A^{r/2}(\bu^0 - \bu_{\alpha, r}^0)}_{L^2}^2 \rp 
    e^{C_Tt} 
    + 
    CT\alpha^{2r}\norm{A^{r/2}\partial_t \bu}_{L^2(0,t;L^2)}^2 e^{C_Tt},
\end{align*}
where $C_T = \frac{C}{T}$.
Therefore, \eqref{convergence_estimate_fNSV_NSE} is proved. Finally, if $\lim\limits_{\alpha\to 0}\norm{\bu^0 - \bu_{\alpha, r}}_{L^2} = 0$ and $\norm{A^{r/2}\bu_{\alpha, r}^0}_{L^2}^2 \leq M < \infty$ for all $\alpha$, taking the $\limsup$ as $\alpha \to 0$ shows that $\bu_{\alpha, r} \to \bu$ in $L^2((0,T), H^1) \cap L^\infty(0,T; H^r)$. 
\end{proof}

\begin{theorem}[Convergence fNSV $\to$ E] \label{thm:convergence_fNSV_E}
Let $s \geq 3$, and $r \in \left[ \frac{1}{2}, 1\right)$. Given initial data $\bu^0, \bu_{\alpha, r}^0\in H^s(\nT^3)$, let $\bu, \bu_{\alpha, r} \in C([0,T], H^s(\nT^3)) \cap C^1([0,T], H^{s-1}(\nT^3))$ be the solutions corresponding to \eqref{NSE_projected} with $\nu = 0$ and \eqref{fNSV_functional} with $\nu > 0$, for $0 < T < T_{\operatorname{max}}$, where $T_{\operatorname{max}}$ is the maximal time for which a solution to the Euler equations exists and is unique. Then: 
\begin{enumerate}[(i)]
    \item For all $t \in [0,T]$, 
    \begin{align}\label{convergence_estimate_fNSV_E}
        & \quad \norm{(\bu - \bu_{\alpha, r})(t)}_{L^2}^2 + \alpha^{2r}\norm{A^{r/2}(\bu - \bu_{\alpha, r})(t)}_{L^2}^2
        \\ &\leq \notag 
        \lp \norm{\bu^0 - \bu_{\alpha, r}^0}_{L^2}^2 
        + 
        \alpha^{2r} \norm{A^{r/2}(\bu^0 - \bu_{\alpha, r}^0)}_{L^2}^2 \rp 
        e^{C_Tt} 
        \\ & \qquad \qquad + \notag 
        \lp CT\alpha^{2r}\norm{A^{r/2}\partial_t \bu(t)}_{L^2(0,t;L^2)}^2 + \nu\norm{A^{1/2}\bu(t)}_{L^2(0,t; L^2)}^2 \rp e^{C_Tt},
   \end{align}
   where $C_T = \frac{C}{T}$ for some constant $C$ and $T$ is a fixed time before potential blow-up.
   \item Consequently, if 
   \begin{align}
       \norm{\bu^0 - \bu_{\alpha, r}^0}_{L^2}^2 + \alpha^{2r}\norm{A^{r/2} (\bu^0 - \bu_{\alpha, r}^0)}_{L^2}^2 \to 0 \text{ as } \alpha \to 0
   \end{align}
   (in particular, when $\bu_{\alpha, r}^0 = \bu^0$), then $\bu_{\alpha, r} \to \bu$ in $C([0,T], H^1)$ as $\alpha, \nu \to 0$. 
\end{enumerate}
\end{theorem}

\begin{proof}
Subtracting \eqref{fNSV_functional} from \eqref{NSE_projected} gives 
\begin{align}\label{convergence_difference_fNSV_E}
    & \quad \partial_t(\bu - \bu_{\alpha, r}) - \alpha^{2r}A^r\partial_t \bu_{\alpha, r} - \nu A\bu_{\alpha, r}
    \\ & = \notag
    -B(\bu, \bu) + B(\bu_{\alpha, r}, \bu_{\alpha, r})
    \\ & = \notag
    -B(\bu - \bu_{\alpha, r}, \bu) + B(\bu - \bu_{\alpha, r}, \bu - \bu_{\alpha, r}) - B(\bu, \bu - \bu_{\alpha, r}).
\end{align}
As in the proof of Theorem \ref{convergence_fEV_E}
\begin{align}\notag 
    & \quad \frac{1}{2} \frac{d}{dt}\lp \norm{\bu - \bu_{\alpha, r}}_{L^2}^2 + \alpha^{2r} \norm{A^{r/2}(\bu - \bu_{\alpha, r})}_{L^2}^2 \rp + \nu (A\bu_{\alpha, r}, \bu - \bu_{\alpha, r}) 
    \\ & = \notag
     -\alpha^{2r}(A^r\partial_t\bu, \bu - \bu_{\alpha, r}) - (B(\bu - \bu_{\alpha, r}, \bu), \bu - \bu_{\alpha, r}).
\end{align}
Adding and subtracting $\nu (A\bu, \bu- \bu_{\alpha, r})$, we obtain 
\begin{align}\label{convergence_IBP_fNSV_E}
    & \quad \frac{1}{2} \frac{d}{dt}\lp \norm{\bu - \bu_{\alpha, r}}_{L^2}^2 + \alpha^{2r} \norm{A^{r/2}(\bu - \bu_{\alpha, r})}_{L^2}^2 \rp + \nu \norm{A^{1/2}(\bu - \bu_{\alpha, r})}_{L^2}^2
    \\ & = \notag
     -\alpha^{2r}(A^r\partial_t\bu, \bu - \bu_{\alpha, r}) - (B(\bu - \bu_{\alpha, r}, \bu), \bu - \bu_{\alpha, r}) - \nu(A\bu, \bu - \bu_{\alpha, r}).
\end{align}
Now using \eqref{convergence_spaces} and integration by parts, we obtain the following estimate 
\begin{align}
    \nu\abs{(A\bu, \bu - \bu_{\alpha, r})} 
    & = 
    \nu\abs{(A^{1/2}\bu, A^{1/2}\bu - \bu_{\alpha, r})} 
    \\ & \leq \notag 
    \nu\norm{A^{1/2}\bu}_{L^2}\norm{A^{1/2}(\bu - \bu_{\alpha, r})}_{L^2}
    \\ & \leq \notag 
    \frac{\nu}{2}\norm{A^{1/2}\bu}_{L^2}^2 + \frac{\nu}{2}\norm{A^{1/2}(\bu - \bu_{\alpha, r})}_{L^2}^2.
\end{align}
This, along with \eqref{fNSV_to_NS_bilinear_bound} and \eqref{fNSV_to_NS_time_derivative_term_bound}, we see that \eqref{convergence_IBP_fNSV_E} can be bounded above as follows 
\begin{align}
    & \quad \frac{1}{2}\frac{d}{dt}\lp \norm{\bu - \bu_{\alpha, r}}_{L^2}^2 + \alpha^{2r} \norm{A^{r/2}(\bu - \bu_{\alpha, r})}_{L^2}^2 \rp + \frac{\nu}{2} \norm{A^{1/2}(\bu - \bu_{\alpha, r})}_{L^2}^2 
    \\ & \leq \notag 
    \frac{C}{2T} \lp \norm{\bu - \bu_{\alpha, r}}_{L^2}^2 
    + 
    \alpha^{2r}\norm{A^{r/2}(\bu - \bu_{\alpha, r})}_{L^2}^2 \rp 
    + 
    \frac{T}{2} \alpha^{2r} \norm{A^{r/2}\partial_t \bu}_{L^2}^2 
    + 
    \frac{\nu}{2}\norm{A^{1/2}\bu}_{L^2}^2.
\end{align}
Gr\"onwall's inequality then yields 
\begin{align*}
    & \quad \norm{(\bu - \bu_{\alpha, r})(t)}_{L^2}^2 + \alpha^{2r} \norm{A^{r/2}(\bu - \bu_{\alpha, r})(t)}_{L^2}^2 + \nu\int_0^t\norm{A^{1/2}(\bu - \bu_{\alpha, r})(s)}_{L^2}^2 \, ds 
    \\ &\leq \notag 
    \lp \norm{\bu^0 - \bu_{\alpha, r}^0}_{L^2}^2 
    + 
    \alpha^{2r} \norm{A^{r/2}(\bu^0 - \bu_{\alpha, r}^0)}_{L^2}^2 \rp 
    e^{C_Tt} 
    \\ & \qquad \qquad + \notag 
    \lp CT\alpha^{2r}\norm{A^{r/2}\partial_t \bu(t)}_{L^2(0,t;L^2)}^2 + \nu\norm{A^{1/2}\bu(t)}_{L^2(0,t; L^2)}^2 \rp e^{C_Tt}.
\end{align*}
Hence, \eqref{convergence_estimate_fNSV_E} is proved. Finally, if $\lim\limits_{\alpha\to 0}\norm{\bu^0 - \bu_{\alpha, r}}_{L^2} = 0$ and $\norm{A^{r/2}\bu_{\alpha, r}^0}_{L^2}^2 \leq M < \infty$ for all $\alpha$, taking the $\limsup$ as $\alpha, \nu \to 0$ shows that $\bu_{\alpha, r} \to \bu$ in $C([0,T), H^1)$. 
\end{proof}

\begin{theorem}[Convergence fNSV $\to$ fEV]
Let $r \in \left(\frac{5}{6},1\right]$, $\alpha >0$, $s > \frac{3}{2}$, and $T_\ast>0$ be fixed. Given initial data $\bu^0, \bu_\nu^0\in H^r(\nT^3)$, let $\bu, \bu_{\nu} \in C([0,T_\ast], V^r(\nT^3)) \cap L^2(0,T_\ast, H^{s}(\nT^3))$ be the solutions corresponding to \eqref{fNSV_functional} with $\nu = 0$ and \eqref{fNSV_functional} with $\nu > 0$, respectively. Then: 
\begin{enumerate}[(i)]
    \item For all $t \in [0,T_\ast]$, 
    \begin{align}
        & \quad \norm{\bu(t) - \bu(t)_{\nu}}_{L^2}^2 + \alpha^{2r}\norm{A^{r/2}(\bu(t) - \bu(t)_{\nu})}_{L^2}^2
        \\ &\leq \notag 
        \lp \norm{\bu^0 - \bu_{\nu}^0}_{L^2}^2 
        + 
        \alpha^{2r} \norm{A^{r/2}(\bu^0 - \bu_{\nu}^0)}_{L^2}^2 \rp 
        e^{C_Tt} 
         + 
         \nu\norm{A^{1/2}\bu}_{L^2(0,t; L^2)}^2 e^{C_Tt},
   \end{align}
   where $C_T = \frac{C}{T}$ for some constant $C$ and $T\in [0, T_\ast]$.
   \item Consequently, if 
   \begin{align}
       \norm{\bu^0 - \bu_{\nu}^0}_{L^2}^2 + \alpha^{2r}\norm{A^{r/2} (\bu^0 - \bu_{\nu}^0)}_{L^2}^2 \to 0 \text{ as } \nu \to 0
   \end{align}
   (in particular, when $\bu_{\nu}^0 = \bu^0$), then $\bu_{\nu} \to \bu$ in $C([0,T], H^1)$ as $\nu \to 0$. 
\end{enumerate}
\end{theorem}

\begin{proof}
Subtracting the fNSV equations from the fEV equations gives 
\begin{align}\notag 
    & \partial_t(\bu - \bu_\nu) - \alpha^{2r}A^r\partial_t(\bu - \bu_\nu) - \nu A\bu_\nu 
    \\ & = \notag 
    -B(\bu - \bu_{\nu}, \bu) + B(\bu - \bu_{\nu}, \bu - \bu_{\nu}) - B(\bu, \bu - \bu_{\nu}).
\end{align}
The rest of the proof follows similarly to that of Theorem \ref{thm:convergence_fNSV_E}.
\end{proof}

 \section{Blow-up Criteria}\label{sec:Blow_up}
 In this section, we analyze singularity formation in the Euler and Navier-Stokes equations. For the Euler equations, the formulation of a singularity can be characterized by the vorticity of the fluid becoming unbounded. In particular, it was proved in \cite{Beale_Kato_Majda_1984} that a singularity develops on the interval $[0,T^{**}]$ if and only if for every $N \in \nN$, there exists $t \in [0, T^{\ast \ast}]$ such that $\norm{\bomega}_{L^1([0,t]; L^\infty)}>N$,
where $\bomega(t) := \nabla \times \bu(t)$. Theorem \ref{thm:blow_up_Euler} was proved in the $r = 1$ case in \cite{Larios_Titi_2009} following the original ideas in \cite{Khouider_Titi_2008}.

\begin{theorem}[Blow-up Criterion for Euler]\label{thm:blow_up_Euler}
Assume $\bu^0 \in H^s$ for some $s \geq 3$. Suppose there exists a finite time $T^{**}>0$ such that solutions $\bu_{\alpha, r}$ of \eqref{fNSV_functional} with $\nu = 0$ with $\bu_{\alpha, r}^0 = \bu^0$ for each $\alpha > 0$ and $r \in \left( \frac{5}{6},1 \right]$ satisfy 
\begin{align}
    \sup\limits_{t \in [0, T^{**})}\limsup\limits_{\alpha \to 0} \alpha^{2r}\norm{A^{r/2} \bu_{\alpha, r}(t)}_{L^2}^2 > 0.
\end{align}
Then the Euler equations \eqref{NSE_projected} with $\nu = 0$, with initial data $\bu^0$, develop a singularity in the interval $[0,T^{**}]$. 
\end{theorem}

\begin{proof}
By way of contradiction, assume the solution $\bu$ of the Euler equations is regular on $[0,T^{**}]$. \dela{Taking an inner product of \eqref{fNSV_functional} (with $\nu = 0$) with $\bu_{\alpha, r}$ and integrating by parts yields } 
By applying the Lions–Magenes lemma (\cite[Lemma III.1.2]{Temam_2001_Th_Num}) we infer that
\begin{align}\notag
    \frac{1}{2}\frac{d}{dt} \lp \norm{\bu_{\alpha, r}}_{L^2}^2 + \alpha^{2r}\norm{A^{r/2}\bu_{\alpha, r}}_{L^2}^2 \rp = 0.
\end{align}
Integrating this equality gives 
\begin{align}\label{modified_energy_equality}
    \norm{\bu_{\alpha, r}(t)}_{L^2}^2 + \alpha^{2r}\norm{A^{r/2}\bu_{\alpha, r}(t)}_{L^2}^2 = \norm{\bu^0}_{L^2}^2 + \alpha^{2r}\norm{A^{r/2}\bu^0}_{L^2}^2,
\end{align}
for all $t \geq 0$. Now, taking a $\limsup$ of \eqref{modified_energy_equality} (and noting that since $\lim_{\alpha\rightarrow0}\norm{\bu_{\alpha, r}(t)}_{L^2}^2$ exists, the $\limsup$ can be applied additively; see Proposition \ref{additivity} the appendix) we obtain 
\begin{align}\label{contradiction}
    \norm{\bu(t)}_{L^2}^2 + \limsup\limits_{\alpha \to 0} \alpha^{2r} \norm{A^{r/2}\bu_{\alpha, r}(t)}_{L^2}^2 = \norm{\bu^0}_{L^2}^2,
\end{align}
for all $t \in [0,T^{**}]$ since $\norm{\bu_{\alpha, r}(t)}_{L^2} \to \norm{\bu(t)}_{L^2}$ for $t \in [0,T^{**}]$ by \eqref{ur_to_u_fEV}. But this is to say that 
\begin{align}
    \limsup\limits_{\alpha \to 0}\alpha^{2r}\norm{A^{r/2}\bu_{\alpha, r}(t)}_{L^2}^2 = 0,
\end{align}
since smooth solutions of the Euler equation conserve energy, i.e. $\norm{\bu(t)}_{L^2} = \norm{\bu^0}_{L^2}$ for $t \in [0,T^{**}]$. Thus, $\bu$ must develop a singularity on $[0,T^{**}]$. 
\end{proof}

\begin{theorem}[Blow-up Criterion for Navier-Stokes]
\label{thm:blow_up_NSE}
Assume $\bu^0 \in H^{r+s}$ for some $s > 1$. Suppose there exists a finite time $T^{**}>0$ such that solutions $\bu_{\alpha, r}$ of \eqref{fNSV_functional} with $\nu > 0$ with $\bu_{\alpha, r}^0 = \bu^0$ for each $\alpha > 0$ and $r \in \left[ \frac{1}{2},1\right]$ satisfy 
\begin{align}
    \sup\limits_{t \in [0, T^{**})}\limsup\limits_{\alpha \to 0} \alpha^{2r}\norm{A^{r/2} \bu_{\alpha, r}(t)}_{L^2}^2 > 0.
\end{align}
Then the Navier-Stokes Equations with initial data in $\bu^0$ develop a singularity on the time interval $[0,T^{**}]$. 
\end{theorem}

\begin{proof}
For the sake of contradiction, assume the solution $\bu$ of the Euler equations is regular on $[0,T^{**}]$. \dela{Taking an inner product of \eqref{fNSV_functional} (with $\nu > 0$) with $\bu_{\alpha, r}$ and integrating by parts gives} Then the following identity holds. 

\begin{align}\notag
    \frac{1}{2}\frac{d}{dt} \lp \norm{\bu_{\alpha, r}}_{L^2}^2 + \alpha^{2r}\norm{A^{r/2}\bu_{\alpha, r}}_{L^2}^2 \rp +\nu\norm{A^{1/2}\bu_{\alpha, r}}_{L^2}^2 = 0.
\end{align}
Integrating this equality gives 
\begin{align}\label{modified_energy_equality_fNSV}
    \norm{\bu_{\alpha, r}(t)}_{L^2}^2 + \alpha^{2r}\norm{A^{r/2}\bu_{\alpha, r}(t)}_{L^2}^2 + 2\nu \int_0^t \norm{A^{1/2} \bu_{\alpha, r}(s)}_{L^2}^2 \, ds = \norm{\bu^0}_{L^2}^2 + \alpha^{2r}\norm{A^{r/2}\bu^0}_{L^2}^2,
\end{align}
for all $t \geq 0$. Now, taking a $\limsup$ of \eqref{modified_energy_equality_fNSV} and using Proposition \ref{additivity} yields 
\begin{align}\label{contradiction_fNSV}
    \norm{\bu(t)}_{L^2}^2 + \limsup\limits_{\alpha \to 0} \alpha^{2r} \norm{A^{r/2}\bu(t)}_{L^2}^2 + \limsup\limits_{\alpha \to 0}2 \nu \int_0^t \norm{A^{1/2}\bu_{\alpha, r}(s)}_{L^2}^2 \, ds = \norm{\bu^0}_{L^2}^2.
\end{align}
for all $t \in [0,T^{**}]$ since $\bu_{\alpha, r} \to \bu$ in $L^2([0,T^{**}], H^1)$ by \eqref{ur_to_u_fNSV}. But this is to say that 
\begin{align}
    \limsup\limits_{\alpha \to 0}\alpha^{2r}\norm{A^{r/2}\bu_{\alpha, r}(t)}_{L^2}^2 = 0,
\end{align}
since strong solutions of the Navier-Stokes equations conserve energy, i.e. \[\norm{\bu(t)}_{L^2}^2 + 2\nu \int_0^t \norm{A^{1/2}\bu}_{L^2}^2 \, ds = \norm{\bu^0}_{L^2}^2,\] 
for $t \in [0,T^{**}]$. Thus, $\bu$ must develop a singularity on the time interval $[0,T^{**}]$. 
\end{proof}

\section{Conclusion}
We have shown that the Voigt technique can be extended to the fractional case, at least in the case of periodic boundary conditions.  This fractional case is of practical interest, as it allows for dissipation, unlike in the case of the standard Voigt regularization.   Moreover, we proved that as the regularization parameter $\alpha\to0$, that solutions to the approximate models converge to those of the original system, indicating that these models may be of use in applications such as turbulence modeling.  We also proved a vanishing viscosity result in the setting of fixed $\alpha$.  In addition, in the spirit of \cite{Khouider_Titi_2008,Larios_Titi_2009} we proved several blow-up criteria for the original systems in terms of Voigt regularization which could be used, e.g., in computational or analytical searches for blow-up of the original equations, as in \cite{Larios_Petersen_Titi_Wingate_2015}.  In future work, we plan to investigate computationally the turbulent statistics of fractional Voigt models, and also the associated blow-up criteria.  Moreover, studying the case of domains with a physical boundary seems non-trivial, but worthwhile, as the fractional Helmholtz operator (at least with $\frac12\leq r\leq1$) seems to require no additional boundary conditions than in the case of no-slip boundary conditions for the Navier-Stokes equations.

\appendix
\section{}
\noindent
Here, we provide a brief justification of the additivity referred to in the proofs of Theorems \ref{thm:blow_up_Euler} and \ref{thm:blow_up_NSE}.  While this proposition is elementary, it appears to be underrepresented in the literature. Furthermore, it has prompted questions during talks by the authors (e.g., whether equations like \eqref{contradiction} and \eqref{contradiction_fNSV} should instead be inequalities). For clarity and completeness, we include it here.
\begin{proposition}\label{additivity}
Suppose the sequence $\set{a_n}_{n\in\nN}$ converges, say $a_n\rightarrow a$, and let $b_n$ be any sequence.  Then
$\limsup_{n\rightarrow\infty}(a_n + b_n)
= a + \limsup_{n\rightarrow\infty}(b_n)$.
\end{proposition}
\begin{proof}
  First, by the subadditive property of the limsup, it holds that \[\limsup_{n\rightarrow\infty}(a_n + b_n)\leq \limsup_{n\rightarrow\infty}(a_n) + \limsup_{n\rightarrow\infty}(b_n) = a + \limsup_{n\rightarrow\infty}(b_n).\]   
  Next, choose a subsequence $b_{n_k}$ such that $\lim_{k\rightarrow\infty}b_{n_k} = \limsup_{n\rightarrow\infty}(b_n)$.  Note that $a_{n_k}\rightarrow a$.  Then $\limsup_{n\rightarrow\infty}(a_n + b_n)\geq \limsup_{k\rightarrow\infty}(a_{n_k} + b_{n_k})
=\lim_{k\rightarrow\infty}(a_{n_k} + b_{n_k})
= a + \limsup_{n\rightarrow\infty}(b_n)$.    
\end{proof}

\section*{Data Availability Statement}
The manuscript has no associated data.

\section*{Conflict of Interest Statement}
On behalf of all authors, the corresponding author states that there is no conflict of interest.

\bigskip 
\hrule 
\bigskip 
\section*{Acknowledgments}
 \noindent
 We would like to thank Prof. Edriss S. Titi for interesting discussions.
Authors Z.B. and A.L. would like to thank the Isaac Newton Institute for Mathematical Sciences, Cambridge, for support and warm hospitality during the programme ``Mathematical aspects of turbulence: where do we stand?'' where work on this paper was undertaken. This work was supported by EPSRC grant no EP/R014604/1.
The research of A.L. and I.S. was supported in part by NSF Grants CMMI-1953346 and DMS-2206741.  A.L. was also supported by USGS Grant No. G23AC00156-01.


\begin{thebibliography}{99}

\bibitem{Hani_2013}
H.~Ali.
\newblock Global well-posedness for deconvolution magnetohydrodynamics models
  with fractional regularization.
\newblock {\em Methods Appl. Anal.}, 20(3):211--235, 2013.

\bibitem{Beale_Kato_Majda_1984}
J.~T. Beale, T.~Kato, and A.~J. Majda.
\newblock Remarks on the breakdown of smooth solutions for the {$3$}-{D}
  {E}uler equations.
\newblock {\em Comm. Math. Phys.}, 94(1):61--66, 1984.

\bibitem{Berselli_Bisconti_2012}
L.~C. Berselli and L.~Bisconti.
\newblock On the structural stability of the {E}uler-{V}oigt and
  {N}avier--{S}tokes--{V}oigt models.
\newblock {\em Nonlinear Anal.}, 75(1):117--130, 2012.

\bibitem{Berselli_Kim_Rebholz_2016}
L.~C. Berselli, T.-Y. Kim, and L.~G. Rebholz.
\newblock Analysis of a reduced-order approximate deconvolution model and its
  interpretation as a {N}avier--{S}tokes--{V}oigt regularization.
\newblock {\em Discrete Contin. Dyn. Syst. Ser. B}, 21(4):1027--1050, 2016.

\bibitem{Berselli_Spirito_2017}
L.~C. Berselli and S.~Spirito.
\newblock Suitable weak solutions to the 3{D} {N}avier--{S}tokes equations are
  constructed with the {V}oigt approximation.
\newblock {\em J. Differential Equations}, 262(5):3285--3316, 2017.

\bibitem{Bisconti_Catania_2021}
L.~Bisconti and D.~Catania.
\newblock On the convergence rates for the three-dimensional filtered
  {B}oussinesq equations.
\newblock {\em Math. Nachr.}, 294(6):1099--1114, 2021.

\bibitem{Bohm_1992}
M.~B{\"o}hm.
\newblock On {N}avier--{S}tokes and {K}elvin-{V}oigt equations in three
  dimensions in interpolation spaces.
\newblock {\em Math. Nachr.}, 155:151--165, 1992.

\bibitem{Borges_Ramos_2013}
R.~Borges and F.~Ramos.
\newblock Sub-grid effects of the {V}oigt viscoelastic regularization of a
  singular dyadic model of turbulence.
\newblock {\em J. Phys. A}, 46(19):195501, 19, 2013.

\bibitem{Brachet_Bustamante_Krstulovic_Mininni_Pouquet_Rosenberg_2013}
M.~E. Brachet, M.~D. Bustamante, G.~Krstulovic, P.~D. Mininni, A.~Pouquet, and
  D.~Rosenberg.
\newblock Ideal evolution of magnetohydrodynamic turbulence when imposing
  {T}aylor-{G}reen symmetries.
\newblock {\em Phys. Rev. E}, 87:013110, Jan 2013.

\bibitem{Brachet_Meiron_Orszag_Nickel_Morf_Frisch_1983_JFM}
M.~E. Brachet, D.~Meiron, S.~Orszag, B.~Nickel, R.~Morf, and U.~Frisch.
\newblock Small-scale structure of the {T}aylor-{G}reen vortex.
\newblock {\em J. Fluid Mech.}, 130:411--452, 1983.

\bibitem{Brezis_2010_functional}
H.~Brezis.
\newblock {\em Functional {A}nalysis, {S}obolev {S}paces and {P}artial
  {D}ifferential {E}quations}.
\newblock Universitext. Springer New York, 2010.

\bibitem{Bustamante_Brachet_2012}
M.~D. Bustamante and M.~Brachet.
\newblock Interplay between the {B}eale-{K}ato-{M}ajda theorem and the
  analyticity-strip method to investigate numerically the incompressible
  {E}uler singularity problem.
\newblock {\em Phys. Rev. E}, 86:066302, Dec 2012.

\bibitem{Cao_Lunasin_Titi_2006}
Y.~Cao, E.~Lunasin, and E.~S. Titi.
\newblock Global well-posedness of the three-dimensional viscous and inviscid
  simplified {B}ardina turbulence models.
\newblock {\em Commun. Math. Sci.}, 4(4):823--848, 2006.

\bibitem{Catania_2009}
D.~Catania.
\newblock Global existence for a regularized magnetohydrodynamic-$\alpha$
  model.
\newblock {\em Ann. Univ Ferrara}, 56:1--20, 2010.
\newblock 10.1007/s11565-009-0069-1.

\bibitem{Catania_Secchi_2009}
D.~Catania and P.~Secchi.
\newblock Global existence for two regularized {MHD} models in three
  space-dimension.
\newblock {\em Portugaliae Mathematica}, 68(1):41--52, 2011.

\bibitem{Chen_Hou_2022_Euler_Blow_up}
J.~Chen and T.~Y. Hou.
\newblock Stable nearly self-similar blowup of the 2{D} {B}oussinesq and 3{D}
  {E}uler equations with smooth data i: Analysis.
\newblock {\em arXiv preprint arXiv:2210.07191}, 2022.

\bibitem{Chen_Foias_Holm_Olson_Titi_Wynne_1998_PRL}
S.~Chen, C.~Foias, D.~D. Holm, E.~Olson, E.~S. Titi, and S.~Wynne.
\newblock Camassa-{H}olm equations as a closure model for turbulent channel and
  pipe flow.
\newblock {\em Phys. Rev. Lett.}, 81(24):5338--5341, 1998.

\bibitem{Chen_Foias_Holm_Olson_Titi_Wynne_1999}
S.~Chen, C.~Foias, D.~D. Holm, E.~Olson, E.~S. Titi, and S.~Wynne.
\newblock The {C}amassa-{H}olm equations and turbulence.
\newblock {\em Phys. D}, 133(1-4):49--65, 1999.
\newblock Predictability: quantifying uncertainty in models of complex
  phenomena (Los Alamos, NM, 1998).

\bibitem{Chen_Foias_Holm_Olson_Titi_Wynne_1998_PF}
S.~Chen, C.~Foias, D.~D. Holm, E.~Olson, E.~S. Titi, and S.~Wynne.
\newblock A connection between the {C}amassa-{H}olm equations and turbulent
  flows in channels and pipes.
\newblock {\em Phys. Fluids}, 11(8):2343--2353, 1999.
\newblock The International Conference on Turbulence (Los Alamos, NM, 1998).

\bibitem{Cheskidov_Holm_Olson_Titi_2005}
A.~Cheskidov, D.~D. Holm, E.~Olson, and E.~S. Titi.
\newblock On a {L}eray-{$\alpha$} model of turbulence.
\newblock {\em Proc. R. Soc. Lond. Ser. A Math. Phys. Eng. Sci.},
  461(2055):629--649, 2005.

\bibitem{Cheskidov_Shvydkoy_2014_unified_blow_up}
A.~Cheskidov and R.~Shvydkoy.
\newblock A unified approach to regularity problems for the 3{D}
  {N}avier--{S}tokes and {E}uler equations: the use of {K}olmogorov's
  dissipation range.
\newblock {\em J. Math. Fluid Mech.}, 16(2):263--273, 2014.

\bibitem{Constantin_Fefferman_1993}
P.~Constantin and C.~Fefferman.
\newblock Direction of vorticity and the problem of global regularity for the
  {N}avier--{S}tokes equations.
\newblock {\em Indiana Univ. Math. J.}, 42(3):775--789, 1993.

\bibitem{Constantin_Foias_1988}
P.~Constantin and C.~Foias.
\newblock {\em Navier--{S}tokes {E}quations}.
\newblock Chicago Lectures in Mathematics. University of Chicago Press,
  Chicago, IL, 1988.

\bibitem{Cuff_Dunca_Manica_Rebholz_2015}
V.~M. Cuff, A.~A. Dunca, C.~C. Manica, and L.~G. Rebholz.
\newblock The reduced order {NS}-{$\alpha$} model for incompressible flow:
  theory, numerical analysis and benchmark testing.
\newblock {\em ESAIM Math. Model. Numer. Anal.}, 49(3):641--662, 2015.

\bibitem{DiMolfetta_Krstlulovic_Brachet_2015}
G.~Di~Molfetta, G.~Krstlulovic, and M.~Brachet.
\newblock Self-truncation and scaling in {E}uler-{V}oigt-$\ensuremath{\alpha}$
  and related fluid models.
\newblock {\em Phys. Rev. E}, 92:013020, Jul 2015.

\bibitem{Doering_Gibbon_1995_book}
C.~R. Doering and J.~D. Gibbon.
\newblock {\em Applied {A}nalysis of the {N}avier--{S}tokes {E}quations}.
\newblock Cambridge Texts in Applied Mathematics. Cambridge University Press,
  Cambridge, 1995.

\bibitem{Donzis_Gibbon_Gupta_Kerr_Pandit_Vincenzi_2013}
D.~A. Donzis, J.~D. Gibbon, A.~Gupta, R.~M. Kerr, R.~Pandit, and D.~Vincenzi.
\newblock Vorticity moments in four numerical simulations of the 3{D}
  {N}avier–{S}tokes equations.
\newblock {\em J. Fluid Mech.}, 732:316–331, 2013.

\bibitem{Ebrahimi_Holst_Lunasin_2012}
M.~A. Ebrahimi, M.~Holst, and E.~Lunasin.
\newblock The {N}avier--{S}tokes--{V}oight model for image inpainting.
\newblock {\em IMA J. App. Math.}, pages 1--26, 2012.
\newblock doi:10.1093/imamat/hxr069.

\bibitem{Ferrari_1993}
A.~B. Ferrari.
\newblock On the blow-up of solutions of the {$3$}-{D} {E}uler equations in a
  bounded domain.
\newblock {\em Comm. Math. Phys.}, 155(2):277--294, 1993.

\bibitem{Foias_Holm_Titi_2002}
C.~Foias, D.~D. Holm, and E.~S. Titi.
\newblock The three-dimensional viscous {C}amassa-{H}olm equations, and their
  relation to the {N}avier--{S}tokes equations and turbulence theory.
\newblock {\em J. Dynam. Differential Equations}, 14(1):1--35, 2002.

\bibitem{Foias_Manley_Rosa_Temam_2001}
C.~Foias, O.~Manley, R.~Rosa, and R.~Temam.
\newblock {\em Navier--{S}tokes {E}quations and {T}urbulence}, volume~83 of
  {\em Encyclopedia of Mathematics and its Applications}.
\newblock Cambridge University Press, Cambridge, 2001.

\bibitem{Gal_Medjo_2013_MMAS}
C.~G. Gal and T.~T. Medjo.
\newblock A {N}avier--{S}tokes--{V}oight model with memory.
\newblock {\em Math. Methods Appl. Sci.}, 36(18):2507--2523, 2013.

\bibitem{Gao_Sun_2012}
H.~Gao and C.~Sun.
\newblock Random dynamics of the 3{D} stochastic {N}avier--{S}tokes--{V}oight
  equations.
\newblock {\em Nonlinear Anal. Real World Appl.}, 13(3):1197--1205, 2012.

\bibitem{Garcia_Luengo_Julia_Read_2012}
J.~Garc{\'{\i}}a-Luengo, P.~Mar{\'{\i}}n-Rubio, and J.~Real.
\newblock Pullback attractors for three-dimensional non-autonomous
  {N}avier--{S}tokes--{V}oigt equations.
\newblock {\em Nonlinearity}, 25(4):905--930, 2012.

\bibitem{Gibbon_Bustamante_Kerr_2008}
J.~D. Gibbon, M.~Bustamante, and R.~M. Kerr.
\newblock The three-dimensional {E}uler equations: singular or non-singular?
\newblock {\em Nonlinearity}, 21(8):T123--T129, 2008.

\bibitem{Gibbon_Titi_2013_Blowup}
J.~D. Gibbon and E.~S. Titi.
\newblock The 3{D} incompressible {E}uler equations with a passive scalar: a
  road to blow-up?
\newblock {\em J. Nonlinear Sci.}, 23(6):993--1000, 2013.

\bibitem{Ignatova_2023_BoussVoigt}
M.~Ignatova.
\newblock 2{D} {V}oigt {B}oussinesq equations.
\newblock {\em J. Math. Fluid Mech.}, 26(1):15, 2024.

\bibitem{Kalantarov_Levant_Titi_2009}
V.~K. Kalantarov, B.~Levant, and E.~S. Titi.
\newblock Gevrey regularity for the attractor of the 3{D}
  {N}avier--{S}tokes--{V}oight equations.
\newblock {\em J. Nonlinear Sci.}, 19(2):133--152, 2009.

\bibitem{Kalantarov_Titi_2009}
V.~K. Kalantarov and E.~S. Titi.
\newblock Global attractors and determining modes for the 3{D}
  {N}avier--{S}tokes--{V}oight equations.
\newblock {\em Chinese Ann. Math. B}, 30(6):697--714, 2009.

\bibitem{Kato_1972}
T.~Kato.
\newblock Nonstationary flows of viscous and ideal fluids in $r^3$.
\newblock {\em J. Functional Analysis}, 9:296--305, 1972.

\bibitem{Kerr_1993}
R.~M. Kerr.
\newblock Evidence for a singularity of the three-dimensional, incompressible
  {E}uler equations.
\newblock {\em Phys. Fluids A}, 5(7):1725--1746, 1993.

\bibitem{Khouider_Titi_2008}
B.~Khouider and E.~S. Titi.
\newblock An inviscid regularization for the surface quasi-geostrophic
  equation.
\newblock {\em Comm. Pure Appl. Math.}, 61(10):1331--1346, 2008.

\bibitem{Khramov_1983}
Y.~A. Khramov.
\newblock {\em \foreignlanguage{russian}{Физики:
  Биографический справочник}}.
\newblock Nauka, Moscow, 2nd edition, revised and supplemented edition, 1983.

\bibitem{Kozono_Ogawa_Taniuchi_2002}
H.~Kozono, T.~Ogawa, and Y.~Taniuchi.
\newblock The critical {S}obolev inequalities in {B}esov spaces and regularity
  criterion to some semi-linear evolution equations.
\newblock {\em Math. Z.}, 242(2):251--278, 2002.

\bibitem{Kozono_Taniuchi_2000}
H.~Kozono and Y.~Taniuchi.
\newblock Limiting case of the {S}obolev inequality in {BMO}, with application
  to the {E}uler equations.
\newblock {\em Comm. Math. Phys.}, 214(1):191--200, 2000.

\bibitem{Kuberry_Larios_Rebholz_Wilson_2012}
P.~Kuberry, A.~Larios, L.~G. Rebholz, and N.~E. Wilson.
\newblock Numerical approximation of the {V}oigt regularization for
  incompressible {N}avier--{S}tokes and magnetohydrodynamic flows.
\newblock {\em Comput. Math. Appl.}, 64(8):2647--2662, 2012.

\bibitem{Larios_Lunasin_Titi_2013}
A.~Larios, E.~Lunasin, and E.~S. Titi.
\newblock Global well-posedness for the 2{D} {B}oussinesq system with
  anisotropic viscosity and without heat diffusion.
\newblock {\em J. Differential Equations}, 255(9):2636--2654, 2013.

\bibitem{Larios_Pei_Rebholz_2018}
A.~Larios, Y.~Pei, and L.~Rebholz.
\newblock Global well-posedness of the velocity-vorticity-{V}oigt model of the
  3{D} {N}avier--{S}tokes equations.
\newblock {\em J. Differential Equations}, 266(5):2435--2465, 2019.

\bibitem{Larios_Petersen_Titi_Wingate_2015}
A.~Larios, M.~R. Petersen, E.~S. Titi, and B.~Wingate.
\newblock A computational investigation of the finite-time blow-up of the 3{D}
  incompressible {E}uler equations based on the {V}oigt regularization.
\newblock {\em Theor. Comput. Fluid Dyn.}, 32(1):23--34, 2018.

\bibitem{Larios_Titi_2009}
A.~Larios and E.~S. Titi.
\newblock On the higher-order global regularity of the inviscid
  {V}oigt-regularization of three-dimensional hydrodynamic models.
\newblock {\em Discrete Contin. Dyn. Syst. Ser. B}, 14(2):603--627, 2010.

\bibitem{Larios_Titi_2010_MHD}
A.~Larios and E.~S. Titi.
\newblock Higher-order global regularity of an inviscid {V}oigt-regularization
  of the three-dimensional inviscid resistive magnetohydrodynamic equations.
\newblock {\em J. Math. Fluid Mech.}, 16(1):59--76, 2014.

\bibitem{Larios_Titi_2015_BC_Blowup}
A.~Larios and E.~S. Titi.
\newblock Global regularity versus finite-time singularities: some paradigms on
  the effect of boundary conditions and certain perturbations.
\newblock {\em Recent progress in the theory of the {E}uler and
  {N}avier--{S}tokes equations}, 430:96--125, 2016.

\bibitem{Layton_Lewandowski_2006}
W.~Layton and R.~Lewandowski.
\newblock On a well-posed turbulence model.
\newblock {\em Discrete Contin. Dyn. Syst. Ser. B}, 6(1):111--128 (electronic),
  2006.

\bibitem{Layton_Rebholz_2013_Voigt}
W.~J. Layton and L.~G. Rebholz.
\newblock On relaxation times in the {N}avier--{S}tokes--{V}oigt model.
\newblock {\em Int. J. Comput. Fluid Dyn.}, 27(3):184--187, 2013.

\bibitem{Levant_Ramos_Titi_2009}
B.~Levant, F.~Ramos, and E.~S. Titi.
\newblock On the statistical properties of the 3{D} incompressible
  {N}avier--{S}tokes--{V}oigt model.
\newblock {\em Commun. Math. Sci.}, 8(1):277--293, 2010.

\bibitem{Li_Qin_2013}
H.~Li and Y.~Qin.
\newblock Pullback attractors for three-dimensional {N}avier--{S}tokes--{V}oigt
  equations with delays.
\newblock {\em Bound. Value Probl.}, pages 2013:191, 17, 2013.

\bibitem{Luo_Hou_2013_Potentially_Singular}
G.~Luo and T.~Y. Hou.
\newblock Potentially singular solutions of the 3{D} axisymmetric {E}uler
  equations.
\newblock {\em Proceedings of the National Academy of Sciences},
  111(36):12968--12973, 2014.

\bibitem{Luo_Hou_2014_Euler_BlowUp}
G.~Luo and T.~Y. Hou.
\newblock Toward the finite-time blowup of the 3{D} axisymmetric {E}uler
  equations: {A} numerical investigation.
\newblock {\em Multiscale Model. Simul.}, 12(4):1722--1776, 2014.

\bibitem{Marchioro_Pulvirenti_1994}
C.~Marchioro and M.~Pulvirenti.
\newblock {\em Mathematical {T}heory of {I}ncompressible {N}onviscous
  {F}luids}, volume~96 of {\em Applied Mathematical Sciences}.
\newblock Springer-Verlag, New York, 1994.

\bibitem{Marsden+Ebin+Fischer_1972}
J.~E. Marsden, D.~G. Ebin, and A.~E. Fischer.
\newblock Diffeomorphism groups, hydrodynamics and relativity.
\newblock In {\em Proceedings of the {T}hirteenth {B}iennial {S}eminar of the
  {C}anadian {M}athematical {C}ongress ({D}alhousie {U}niv., {H}alifax,
  {N}.{S}., 1971), {V}ol. 1}, pages 135--279. Canad. Math. Congr., Montreal,
  QC, 1972.

\bibitem{Morf_Orszag_Frisch_1980_computational_blow_up}
R.~H. Morf, S.~A. Orszag, and U.~Frisch.
\newblock Spontaneous singularity in three-dimensional inviscid, incompressible
  flow.
\newblock {\em Phys. Rev. Lett.}, 44:572--575, Mar 1980.

\bibitem{Mulungye_Lucas_Bustamante_2015}
R.~M. Mulungye, D.~Lucas, and M.~D. Bustamante.
\newblock Symmetry-plane model of 3{D} {E}uler flows and mapping to regular
  systems to improve blowup assessment using numerical and analytical
  solutions.
\newblock {\em J. Fluid Mech.}, 2015.

\bibitem{Niche_2016_JDE}
C.~J. Niche.
\newblock Decay characterization of solutions to {N}avier--{S}tokes--{V}oigt
  equations in terms of the initial datum.
\newblock {\em J. Differential Equations}, 260(5):4440--4453, 2016.

\bibitem{Oskolkov_1973}
A.~P. Oskolkov.
\newblock The uniqueness and solvability in the large of boundary value
  problems for the equations of motion of aqueous solutions of polymers.
\newblock {\em Zap. Nau\v cn. Sem. Leningrad. Otdel. Mat. Inst. Steklov.
  (LOMI)}, 38:98--136, 1973.
\newblock {B}oundary value problems of mathematical physics and related
  questions in the theory of functions, 7.

\bibitem{Oskolkov_1976}
A.~P. Oskolkov.
\newblock Some nonstationary linear and quasilinear systems that arise in the
  study of the motion of viscous fluids.
\newblock {\em Zap. Nau\v cn. Sem. Leningrad. Otdel. Mat. Inst. Steklov.
  (LOMI)}, 59:133--177, 257, 1976.
\newblock {B}oundary value problems of mathematical physics and related
  questions in the theory of functions, 9.

\bibitem{Oskolkov_1982}
A.~P. Oskolkov.
\newblock On the theory of unsteady flows of {K}elvin-{V}oigt fluids.
\newblock {\em Zap. Nauchn. Sem. Leningrad. Otdel. Mat. Inst. Steklov. (LOMI)},
  115:191--202, 310, 1982.
\newblock {B}oundary value problems of mathematical physics and related
  questions in the theory of functions, 14.

\bibitem{Pei_2021}
Y.~Pei.
\newblock Regularity and convergence results of the velocity-vorticity-{V}oigt
  model of the 3{D} {B}oussinesq equations.
\newblock {\em Acta Appl. Math.}, 176:Paper No. 8, 25, 2021.

\bibitem{Ponce_1985}
G.~Ponce.
\newblock Remarks on a paper: ``{R}emarks on the breakdown of smooth solutions
  for the 3-{D} {E}uler equations'' [{C}omm. {M}ath. {P}hys. {\bf 94} (1984),
  no. 1, 61--66; {MR}0763762 (85j:35154)] by {J}. {T}. {B}eale, {T}. {K}ato and
  {A}. {M}ajda.
\newblock {\em Comm. Math. Phys.}, 98(3):349--353, 1985.

\bibitem{Yuming_Qin_2025_DCDS}
Y.~Qin and H.~Jiang.
\newblock Existence of pullback attractors and invariant measures for 3d
  {N}avier-{S}tokes-{V}oigt equations with delay.
\newblock {\em Discrete and Continuous Dynamical Systems - B}, 30(1):243--264,
  2025.

\bibitem{Ramos_Titi_2010}
F.~Ramos and E.~S. Titi.
\newblock Invariant measures for the 3{D} {N}avier--{S}tokes--{V}oigt equations
  and their {N}avier--{S}tokes limit.
\newblock {\em Discrete Contin. Dyn. Syst.}, 28(1):375--403, 2010.

\bibitem{robinson_Rodrigo_Sadowski_3DNSE_2016}
J.~C. Robinson, J.~L. Rodrigo, and W.~Sadowski.
\newblock {\em The {T}hree-{D}imensional {N}avier--{S}tokes {E}quations},
  volume 157 of {\em Cambridge Studies in Advanced Mathematics}.
\newblock Cambridge University Press, Cambridge, 2016.
\newblock Classical theory.

\bibitem{Rong_Fiodilino_Shi_Cao_2022}
Y.~Rong, J.~A. Fiordilino, F.~Shi, and Y.~Cao.
\newblock A modular {V}oigt regularization of the {C}rank--{N}icolson finite
  element method for the {N}avier--{S}tokes equations.
\newblock {\em J. Sci. Comput.}, 92(3):Paper No. 101, 35, 2022.

\bibitem{Tang_2014}
Q.~B. Tang.
\newblock Dynamics of stochastic three dimensional {N}avier--{S}tokes--{V}oigt
  equations on unbounded domains.
\newblock {\em J. Math. Anal. Appl.}, 419(1):583--605, 2014.

\bibitem{Temam_1976}
R.~Temam.
\newblock Local existence of {$C\sp{\infty }$} solutions of the {E}uler
  equations of incompressible perfect fluids.
\newblock In {\em Turbulence and {N}avier--{S}tokes equations ({P}roc. {C}onf.,
  {U}niv. {P}aris-{S}ud, {O}rsay, 1975)}, pages 184--194. Lecture Notes in
  Math., Vol. 565. Springer, Berlin, 1976.

\bibitem{Temam_1995_Fun_Anal}
R.~Temam.
\newblock {\em Navier--{S}tokes {E}quations and {N}onlinear {F}unctional
  {A}nalysis}, volume~66 of {\em CBMS-NSF Regional Conference Series in Applied
  Mathematics}.
\newblock Society for Industrial and Applied Mathematics (SIAM), Philadelphia,
  PA, second edition, 1995.

\bibitem{Temam_2001_Th_Num}
R.~Temam.
\newblock {\em Navier--{S}tokes {E}quations: {T}heory and {N}umerical
  {A}nalysis}.
\newblock AMS Chelsea Publishing, Providence, RI, 2001.
\newblock Theory and numerical analysis, Reprint of the 1984 edition.

\bibitem{Zang_2022}
A.~Zang.
\newblock Local well-posedness for boundary layer equations of {E}uler-{V}oigt
  equations in analytic setting.
\newblock {\em J. Differential Equations}, 307:1--28, 2022.

\bibitem{Zhao_Zhu_2015}
C.~Zhao and H.~Zhu.
\newblock Upper bound of decay rate for solutions to the
  {N}avier--{S}tokes--{V}oigt equations in {$\mathbb{R}^3$}.
\newblock {\em Appl. Math. Comput.}, 256:183--191, 2015.

\end{thebibliography}
\end{document}